%% file: main.tex
\documentclass[letterpaper,11pt,reqno]{amsart} 
\usepackage[portrait,margin=3cm]{geometry}
\include{style}

\renewcommand{\odot}{\circ}
\newcommand{\dout}{d^{\mathrm{out}}}

\newtheorem{theorem}{Theorem}[section]

\newtheorem{remark}[theorem]{Remark}

\usepackage{tikz}
\usetikzlibrary{arrows.meta}

\begin{document}

\title[Joint estimations for spin glasses]{Joint parameter estimations for spin glasses}
\author[Chen]{Wei-Kuo Chen$^\star$}
\author[Sen]{Arnab Sen$^\dagger$}
\author[Wu]{Qiang Wu$^\ddagger$}
\address{School of Mathematics, University of Minnesota, 127 Vincent Hall 206 Church St. SE Minneapolis, MN 55455}
\email{$^\star$wkchen@umn.edu}
\email{$^\dagger$arnab@umn.edu}
\email{$^\ddagger$wuq@umn.edu}
\date{\today}
\thanks{}
\begin{abstract}
	Spin glass models with quadratic-type Hamiltonians are disordered statistical physics systems with competing ferromagnetic and anti-ferromagnetic spin interactions. The corresponding Gibbs measures belong to the exponential family parametrized by (inverse) temperature $\gb>0$ and external field $h\in\bR$. Given a sample from these Gibbs measures, a statistically fundamental question is to infer the temperature and external field parameters.
 In 2007, Chatterjee~(Ann.~Statist.~35 (2007),~no.~5,~1931–1946) first proved that in the absence of external field $h=0$, the maximum pseudolikelihood estimator for $\gb$ is $\sqrt{N}$-consistent under some mild assumptions on the disorder matrices. 
 It was left open whether the same method can be used to estimate the temperature and external field simultaneously. 
 In this paper, under some easily verifiable conditions, we prove that the bivariate maximum pseudolikelihood estimator is indeed jointly $\sqrt{N}$-consistent for the temperature and external field parameters. The examples cover the classical Sherrington-Kirkpatrick model and its diluted variants. 
 

\end{abstract}

\maketitle

\section{Introduction and Main Results}

Ising model possesses the Gibbs measure that belongs to the bivariate exponential family and takes the form, for $\mvgs\in \{-1,+1\}^N$, 
\[
P_{\gb,h}(\mvgs) =  \frac{1}{Z_N(\beta,h)}\exp(\gb  \la\mvgs, J\mvgs \ra/2 +h \la \mvgs, \mv1 \ra ),
\]
where $\gb>0$ is the (inverse) temperature, $h\in \bR$ is the external field, and $J\in \bR^N\times \bR^N$ is a real-valued coupling matrix. In the past decades, estimation problems for the Ising model have received a great deal of attention from various communities due to its broad downstream applications, such as in statistics and computer science etc. One of the major research branches is, if the coupling matrix $J$ is known, about estimating $\gb$, $h$, or both from a given sample of the Gibbs measure. One classical approach to this estimation problem is based on the maximum likelihood method. In the case of mean field Ising model, also known as Curie-Weiss model, where $J$ is the adjacency matrix of complete graph with proper scaling, Comets and Gidas~\cite{CG91} showed that for $\gb>0$ and $h\neq 0$, the maximum likelihood estimator (MLE) for $\gb$ given $h$ or vice versa is always $\sqrt{N}$-consistent. However, joint estimation of $(\gb,h)$ is impossible. In some cases, maximum likelihood estimation for the Ising model can be very challenging since it involves the normalizing constant $Z_N(\gb,h)$ whose computation, in particular, is nontrivial for models with $J$ having both positive and negative entries. For example, in the setting of spin glass models, the interaction matrix $J$ consists of random entries; as a result, understanding the limiting behavior of the log-normalizing constant as $N\to \infty$ is generally a challenging problem and only a few cases are solvable in the literature. Arguably, one of the most intriguing results was the famous Parisi formula~\cite{Par79,Par80} that provided a variational representation for this limit established in the seminal work of Talagand \cite{Tal06}. Nonetheless, it is still a mystery to understand spin glass models for large but finite $N$.

Due to the intractability of MLE approach, it is desirable to seek for some other candidate methods that do not involve computing the log-normalizing constants. To this end, the pseudolikelihood approach~\cite{Bes74,Bes75} was proposed, where the pseudolikelihood function is defined as the product of conditional distributions such that the normalizing constant gets cancelled out. In this case, the maximum pseudolikelihood estimator (MPLE) can indeed alleviate the issue of dealing with a possibly complicated normalizing constant.  

The MPLE approach has been studied a lot recently due to its efficiency and simplicity. The first work for the Ising-type model was initiated by Chatterjee~\cite{Chat07}, where the author set $h=0$ and considered the single-parameter estimation problem for $\gb>0$. Under the assumption that $J$ has bounded operator norm 
and $\liminf_{N\to\infty}N^{-1}\ln Z_N(\beta,h)>0$, it was shown that the MPLE is $\sqrt{N}$-consistency estimation for $\gb$. In particular, the result includes examples of spin glasses, such as the Sherrington-Kirkpatrick (SK) model~\cite{SK75} and the Hopfield model~\cite{Hop82}. Afterwards, the result of~\cite{Chat07} was extended to other generalized Ising models~\cite{BM18,MSB22}. All those results have been restricted to the single parameter estimation setting. As we saw in the Curie-Weiss model, joint estimation is generally a much more challenging task. In fact, sometimes MPLE does not be even exist as opposed to the scenario for the single parameter. Incidentally, extending the single parameter estimation results to the joint parameters estimation setting was proposed as one of the open problems by Chatterjee~\cite{Chat07}. An attempt to resolve this question was obtained by Ghosal and Mukherjee \cite{GM20} in the setting of the ferromagnetic Ising model where $J$ has nonnegative entries and they established a set of sufficient conditions under which joint estimation for $(\gb,h)$ is possible with $\sqrt{N}$-consistency. 

 In this paper, we aim to show that MPLE is a $\sqrt{N}$-jointly consistent estimator for $(\beta,h)$ allowing $J$ to possess positive and negative entries. More precisely, the matrix $J$ is a weighted adjacency matrix of a graph, where the weights are i.i.d.\ centered random variable with finite third moment and the graph requires that its averaging degree diverges at least logarithmically fast and is comparable to the maximum degree. This setting covers a collection of spin glass models such as the celebrated SK model and its diluted variants. We show that the validity of our result requires checking two major conditions: the boundedness of the operator norm of $J$, and the positivity of the Hessian matrix of the pseudolikelihood function, see Theorem \ref{thm:exs-cons}. Among them, the second condition plays a central role and it also appeared in \cite{GM20}. Incidentally, the authors therein verified this positivity condition critically relying on the nonnegativity assumption of coupling coefficients which allowed them to use the nonlinear large deviation theory for random graphs~\cite{CD16,BM17,EG18,Chat17}. However, such technique does not seem to work when the coupling matrix $J$ has both positive and negative entries. Until now, it remained as a main obstacle to extend joint estimation results beyond the non-negative coupling coefficients setting. We are able to check the non-singularity condition with high probability in the setting of random coupling matrices via
 the small ball probability arguments~\cite{RV08, Ver14}.
 
\subsection{Notation}
Let $\cG:=([N], E)$ be a graph with vertex set $[N]=\{1,2,\ldots,N\}$ and edge set $E$. Denote by $d_i$ and $d$ the \emph{degree} of vertex $i\in[N]$ and the \emph{average degree} of $\cG$, respectively. In addition, $d_{\max}: = \max_{i\in[N]} d_i$ denotes the \emph{maximum degree} of $\cG$. Let $G=(g_{ij})_{i,j=1}^N$ be a symmetric matrix and vanishes along the diagonal, \ie~$g_{ii} = 0$ for $i\in[N]$. Set \begin{align}\label{eq:J-form}
    J:= \frac{1}{\sqrt{d}} A \odot G,
\end{align}
where $\circ$ denotes the Hadmard product and 
$A$ is the adjacency matrix of $\cG$, that is,
\begin{align*}
    A_{ij} = \begin{cases}
        1, \quad \text{if $(i,j)\in E$}, \\
        0, \quad \text{if $i=j$ or $(i,j) \notin E$. }
    \end{cases}
\end{align*}

If $U_N$ and  $V_N$ are two non-negative random variables sampled from some probability measure $P$, we say $U_N = O_p(V_N)$ if $(U_N/V_N)_{N \ge 1}$ is a tight family under $P$, i.e., for each $\epsilon>0$, there exists $C$ such that 
\[ P( U_N/ V_N \le C) \ge  1- \epsilon \ \ \text{ for all } N \ge 1. \]
The notation $\Omega_p$ is defined similarly. 

\subsection{Basic Setting of MPLE}
For any $\gb>0, h \in \mathbb{R}$, the Gibbs measure associated to $J$ is defined as
\begin{align}\label{eq:model-def}
    P_{\gb,h}(\mvgs) = \frac{\exp(\gb  \la J\mvgs, \mvgs \ra/2 + h \la \mvgs, \mv1 \ra)}{Z_N(\gb,h)},
\end{align}
for all $\mvgs \in \gS_N:=\{-1,+1\}^N$,
where \[
Z_N(\gb,h)=\sum_{\mvgs \in \gS_N} \exp(\gb   \la J\mvgs, \mvgs \ra/2 + h\la \mvgs, \mv1 \ra)\]
is the normalizing constant, also known as the partition function in statistical physics. For any $\tau\sim P_{\beta,h}$ and $f:\Sigma_N\to\mathbb{R}$, denote by $\la f(\tau)\ra$ the expectation of $f(\tau)$ with respect to the Gibbs measure. The free energy is defined as 
\begin{align}
    F_N(\gb,h) = \frac1N \log Z_N(\gb,h).
\end{align}

We proceed to formulate the MPLE for $(\beta,h).$ In the sequel, we assume that $\gb_0,h_0$ are the ground truth parameters and a sample $\mvgt \sim P_{\gb_0,h_0}(\cdot)$. For $\mvgs \in \{-1,+1\}^N$, the \emph{pseudolikelihood function} is defined as
\begin{align}
  \nonumber  L(\gb, h| \mvgs)& = \log \prod_{i=1}^N P_{\gb, h}(\gs_i | \gs_{j\neq i}) \\
   \label{eq:likeli} & = \sum_{i=1}^N \gs_i \Bigl(\gb \sum_{j=1}^N  J_{ij}\gs_j + h\Bigr) -\sum_{i=1}^N \log \cosh \Bigl(\gb \sum_{j\neq i} J_{ij}\gs_j + h\Bigr) - \log 2.
\end{align}
Taking the partial derivatives of the pseudolikelihood function w.r.t.\ the parameters $\gb,h$, we have the corresponding \emph{score functions},
\begin{align}\label{eq:score}
S(\gb,h|\mvgs)  & :=  \frac{\partial L}{\partial \gb} = \sum_{i=1}^N m_i(\mvgs) ( \gs_i - \tanh(\gb m_i(\mvgs)+h)), \\
Q(\gb,h|\mvgs) & := \frac{\partial L}{\partial h} = \sum_{i=1}^N  ( \gs_i - \tanh(\gb m_i(\mvgs)+h)),\label{eq:score2}
\end{align}
where $$
m_i(\mvgs) := \sum_{j=1}^N J_{ij} \gs_j = \sum_{j\neq i} J_{ij} \gs_j.$$ 
We further compute the negative Hessian matrix $H(\gb,h|\mvgs)$ for later use as
\begin{align}\label{eq:hess}
  H(\gb, h|\mvgs) := \begin{pmatrix}
        &  \sum_{i=1}^N m_i(\mvgs)^2 \theta_i(\gb,h|\mvgs) &  \sum_{i=1}^N m_i(\mvgs) \theta_i(\gb,h|\mvgs) \\
        &  \sum_{i=1}^N m_i(\mvgs) \theta_i(\gb,h|\mvgs)&  \sum_{i=1}^N \theta_i(\gb,h|\mvgs)
    \end{pmatrix},
\end{align}
where for $i\in [N],$
$
\theta_i(\gb,h|\mvgs) := \sech^2(\gb m_i(\mvgs) +h).
$
From a straightforward computation, the determinant of the Hessian is 
\[
\abs{H(\gb,h|\mvgs)} = \frac{N^2}{2}  \widetilde{T}_N(\gb,h|\mvgs),
\]
where
\begin{align}\label{eq:det-hess}
    \widetilde{T}_N(\gb,h|\mvgs) := \frac{1}{N^2} \sum_{i,j\in[N]} \theta_i(\mvgs)\theta_j(\mvgs) (m_i(\mvgs)-m_j(\mvgs))^2.
\end{align}
A quantity accompanied with $\widetilde T_N(\gb,h| \mvgs)$ is 
\begin{align}
  T_N(\gb,h|\mvgs) := \frac{1}{N^2}\sum_{i,j\in[N]} (m_i(\mvgs) - m_j(\mvgs))^2 = \frac2N \sum_{i\in [N]}(m_i(\mvgs) - \bar{\mvm}(\mvgs))^2,
\end{align}
where $$\bar{\mvm}(\mvgs):= \frac1N \sum_{i\in [N]} m_i(\mvgs).$$
For notational brevity, we will often suppress the dependence of $\theta_i, T_N,$ and $\widetilde T_N$ on $\beta$ and $h$.

We assume that the coupling matrix $J$ is given and we have access to a single sample from the ground truth $P_{\gb_0,h_0}$. Our ultimate aim is to determine the conditions on the model so that the pseudolikelihood estimator 
exists and is consistent. The following theorem presents a set of sufficient conditions to attend this purpose.

\begin{thm}[Existence and consistency of MPLE]\label{thm:exs-cons}
   Suppose that the coupling matrix $J\in \bR^{N\times N}$ is symmetric and has zeros entries along the diagonal. For fixed $\beta_0>0$ and $h_0\in \mathbb{R}$, let  $\mvgt\in \{-1,+1\}^N$ be a sample drawn from the Gibbs measure $P_{\gb_0,h_0}$. Assume that the following conditions hold. 
   \begin{enumerate}
     \item[$(i)$]  Boundedness of $J$$:$
     \begin{align*} 
     \| J\|  = O(1), \quad \text{where $\|J\|$ is the operator norm of $J$.}
     \end{align*}
          
    \item[$(ii)$] Positivity of the Hessian\,$:$ Under $P_{\gb_0,h_0}$, 
     \begin{align*} 
     \widetilde{T}_N(\mvgt) =  \widetilde{T}_N(\beta_0, h_0| \mvgt) = \Omega_p(1).
     \end{align*}
   \end{enumerate}
   Then with $P_{\gb_0,h_0}$-probability tending to one, the MPLE exists, i.e., the bivariate equation 
   \begin{align}\label{eq:MPLE}
    (S(\gb,h|\mvgt), Q(\gb,h|\mvgt)) = (0,0),
\end{align}
has a unique solution $(\hat \beta(\mvgt), \hat h(\mvgt))  \in \mathbb{R}^2$. Moreover, $(\hat \beta(\mvgt), \hat h(\mvgt))$ is a $\sqrt{N}$-consistent estimator of the ground truth $(\gb_0,h_0)$ in the sense that under $P_{\gb_0,h_0}$, 
    \[
    \| (\hat{\gb}(\mvgt) - \gb_0, \hat{h}(\mvgt) - h_0) \|_2 = O_p(N^{-1/2}).
    \]
 \end{thm}

Note that this theorem holds for arbitrary real-valued symmetric matrices $J$ with zeros along the diagonal and they do not need to be of the form \eqref{eq:J-form}. We further remark that the results can be easily extended to include the regime $\gb_0<0$. For notational and technical convenience, we focus on the regime $\gb_0>0$ as stated above.

\begin{rem}\label{remark1.1}
\rm
   Under different assumptions, Ghosal-Mukherjee \cite{GM20} obtained similar existence and consistency results as Theorem 1.2, where the matrix $J$ has nonnegative entries and uniformly bounded row sums. 
    Our proof for the consistency part in Theorem~\ref{thm:exs-cons} is similar to that of Theorem 1.2 in~\cite{GM20}.   However, the approach for the existence part in the same theorem does not apply to our setting.
    
\end{rem}

\begin{remark}
    \rm
If $h=0$, Theorem \ref{thm:exs-cons} reduces to the single parameter estimation problem for $\beta$ studied in \cite{Chat07} and it was known that the conditions (i) and (ii) imply the invertibility of the corresponding score function analogous to (iii). However, in the joint estimation problem, (i) and (ii) are not enough to imply (iii) as can be seen in the Curie-Weiss model.

\end{remark}

\subsection{Main Results}

Our work considers the setting of the disordered spin system, namely, $G$ is a symmetric matrix of independent random entries. The main result establishes three major theorems that validate the conditions (i), (ii), and (iii) in Theorem~\ref{thm:exs-cons} individually with high probability. As a consequence, we conclude that the pseudolikelihood estimator exists and is $\sqrt{N}$-consistent in the joint estimation problem. To state our main results, we begin by introducing  the following assumption on the graph structure $\mathcal{G}$.

\begin{ass}\label{ass:graph}
    For the underlying graph $\cG = ([N],E)$, we assume the average degree $d$ and the maximum degree $d_{\max}$ satisfy 
    \begin{enumerate}
        \item[$(i)$] $d \gg \log N$,
        \item[$(ii)$] $d_{\max}\le Cd$, 
    \end{enumerate}
    where  $C>0$ is a constant independent of $N.$
\end{ass}

Heuristically, the above condition specifies that the underlying graph can neither be too sparse nor have some vertices with very large degrees. Ghosal and Mukherjee's work focused on non-disordered systems, where their coupling matrix $J$ is defined as  $J  = \frac{1}{d} \cdot A$, with $A$ being the adjacency matrix of the underlying graph and $d$ the average degree. This contrasts with our setup, which employs a scaling factor of $\frac{1}{\sqrt{d}}$ and assigns independent, mean-zero random weights to each edge. Their paper operates under the standing assumption of row sum boundedness for $J$, implying $d_{max} \le C d,$ a condition we also assume.
    Furthermore, their results  show that 
    \begin{itemize}
        \item[(i)] If $d = O(1)$, the MPLE is $\sqrt{N}$-consistent. 
        \item[(ii)] If $d \to \infty$,  depending on whether the underlying graph is irregular or regular, the MPLE is  either $\sqrt{N}$-consistent and the estimation is not possible respectively. 
    \end{itemize}
    To compare with our results, we do not have any results in the regime $d = O(1)$ due to the our standing assumption $d\gg \log N$. Second, under assumption  $d \gg \log N$, we do not need any conditions on the regularity of the underlying graph.

Recall that a random variable $X$ is $\kappa^2$-subgaussian if $\mathbb{E}X=0$ and $\mathbb{E} e^{sX}\leq e^{\kappa^2 s^2/2}$ for all $s\in \mathbb{R}.$ We say that  $G=(g_{ij})_{i,j\in[N]}$ is a symmetric $\kappa^2$-subgaussian matrix with vanishing diagonal if its upper triangular entries are independent $\kappa^2$-subgaussian. Now we state the first result that verifies the boundedness of the operator norm of $J$.

\begin{thm}\label{thm:J-norm}
Suppose that $G$ is a symmetric $\kappa^2$-subgaussian matrix with vanishing diagonal and Assumption~\ref{ass:graph} holds. 
Then there exist constants $C,K>0$, depending on~$\kappa$, such that
\[
\bP(\norm{J} \ge C) \le K \exp(-d/\log N).
\]

\end{thm}

As we shall see below, the proofs for the above theorem are not too difficult, which borrow some standard ideas and results from the concentration inequalities for sparse random matrices \cite{BvH16,CDG23}. The much more difficult component of this work lies on the verification of the strict positivity of the Hessian stated below.

\begin{thm}\label{thm:TN-check}
Suppose that 
$G$ is a symmetric random matrix with vanishing diagonal so that its upper triangular entries satisfy
        \begin{align}\label{eq:disorder-cond}
        \text{$(g_{ij})_{i>j}$ are independent with} \ \E{g_{ij}} = 0,\,{\E[g_{ij}^2]}=1,\,\text{and}\,\E|g_{ij}|^3 \le B,
        \end{align}
where $B>0$ is some finite constant independent of $N$. Furthermore, suppose that Assumption \ref{ass:graph} $(i)$ is satisfied and $d\to\infty$ as $N\to\infty$. Then there exist positive constants $c_0,c$, and $K$ such that
    \[
    \bP\bigl(\min_{\mvgs}\widetilde{T}_N(\mvgs) \le c_0\bigr) \le K \exp( - cN).
    \] 
\end{thm}

For the sample $\mvgt$ from ground truth $P_{\gb_0,h_0}$, obviously $\widetilde{T}_N(\mvgt) \ge  \min_{\mvgs}\widetilde{T}_N(\mvgs)$ holds. The above theorem suggests that $\widetilde{T}_N(\mvgt) \le c_0$ holds with exponentially decaying annealed probability. 

\begin{rem}
    \rm
If the minimization is over a continuous sphere $S_N:=\{\mvgs\in \mathbb{R}^N:\|\mvgs\|^2=N\}$ instead of discrete cube $\{-1,+1\}^N$, then the statement of Theorem~\ref{thm:TN-check} is not true any more. 
For example,  if  $J$ is nonsingular with probability one,  which, in particular, holds for the SK model, then we take $\mvgs^*  =\gamma J^{-1} \mathbf{1} $ for $\gamma := N^{-1/2}\norm{J^{-1} \mathbf{1}}$, 
    so that $m(\mvgs^*) =  \gamma \mathbf{1}$ and consequently,
    \[\widetilde{T}_N(\mvgs^*) \le T_N(\mvgs^*)  = 0.\]
    Therefore, we always have
    \[
    \min_{\mvgs \in S_N} \widetilde{T}_N(\mvgs) = 0.
    \]    
\end{rem}

As of now, with high probability, all conditions in Theorem~\ref{thm:exs-cons} have been verified under varying assumptions. Putting them all together, we obtain the asserted joint estimation for $(\beta,h).$

\begin{thm}\label{thm:conclu}
   For $\gb_0>0$ and $h_0 \in \mathbb{R}$, suppose $\mvgt \sim P_{\gb_0,h_0}$. Suppose that 
   \begin{enumerate}
        \item[$(1)$] $G$ is a symmetric $\kappa^2$-subgaussian random matrix with vanishing diagonal,
        \item[$(2)$] Assumption~\ref{ass:graph} holds.
   \end{enumerate}
   Then with $\mathbb{P}$-probability tending to one, the bivariate MPLE, $(\hat{\gb}(\mvgt),\hat{h}(\mvgt)),$ is jointly $\sqrt{N}$-consistent.
\end{thm}

 
\begin{ex}\rm 
Letting the upper triangular entries of $G$ being i.i.d. standard Gaussian, Theorem \ref{thm:conclu} covers the classical SK model, for which $\cG$ is a complete graph so that $d_{\max} = d = N-1 \gg \log N$ satisfying Assumption~\ref{ass:graph}. Another example is the diluted SK model, where $\cG$ is a sparse Erd\"os-R\'enyi random graph $\cG=G(N,p_N)$ with $p_N \gg {\log N}/{N}$.
\end{ex}

\begin{rem}\rm
        It was indicated in \cite{GM20} that the MPLE estimators are not jointly consistent when $J$ is the scaled adjacency matrix of any asymptotically regular graph $\mathcal{G}$ with the average degree diverging to infinity. In particular, consider the case when  $\mathcal{G}=G(N,p_N)$ is a dense Erd\"os-R\'enyi random graph with $p_N=p$ for some fixed $p\in (0,1]$.
        For $t>0$, let $\Theta_t = \{ (\beta, h) \in (0, \infty)^2: t = \tanh(\beta t+ h)\}.$
        It was shown in \cite[Theorem 1.6]{GM20} that for any $(\beta_0,h_0)\in \Theta_t$, the measure $\mathbb{Q}_n \times \mathcal{G} $ is contiguous to $P_{\beta_0, h_0} \times \mathcal{G} $, where $\mathbb{Q}_n$ is the product of Rademacher measures on $\{-1, 1\}^N$ with mean $\tanh (t)$. Now, note that for any $(\beta_0,h_0)\in (0,\infty)^2$, there exists a unique $t\in (0,1)$ that satisfies $t=\tanh(\beta_0t+h_0)$ and that  one can always find a distinct $(\beta_1,h_1)\in \Theta_t.$ This statement readily implies that there does not exist any estimators that are jointly consistent for any $(\beta_0,h_0)\in (0,\infty)^2$. However, for the same interaction graph $\mathcal{G}$, if $J$ is of the form as  \eqref{eq:J-form}, then we know from Theorem \ref{thm:conclu} that the MPLE are jointly consistent for any $\beta_0>0$ and $h_0 \in \mathbb{R}$. This example suggests that the estimation problem when $J$ has only non-negative entries is significantly different from the case when $J$ is allowed to have both positive and negative entries. 
\end{rem}

\subsection{Proof Sketches}\label{ssec:pf-idea}
    Theorem \ref{thm:exs-cons} consists of existence and consistency parts. The latter follows from a standard Delta method argument similar to  \cite[Theorem 1.2]{GM20}. For existence, note that if the pseudolikelihood function diverges to the negative infinity as $|\beta|+|h|\to\infty$, then it has a unique maximizer (i.e., the MPLE exists) as it is already strictly concave thanks to the condition $(ii)$ in Theorem \ref{thm:exs-cons}.  Almost the entirety of the proof of existence part is devoted to showing this divergence. We prove that this is indeed true if $(\tau_i,m_i(\tau))_{i\in [N]}$ contains enough different signs. More precisely, the sufficient condition we need is stated in Lemma~\ref{fix:prop} and we then show that this property holds w.h.p.\ under the assumptions $(i)$ and $(ii)$. Even though this is intuitively clear since the spins and the fields $m_i$ should not be in favor of certain choices of signs overwhelmingly, proving it rigorously takes some effort.  The key role is played by a general concentration result, see Lemma \ref{lem:new-concent} below, which says that for $\tau\sim P_{\beta,h}$ and any suitably Lipschitz function $f$, the linear statistics  $N^{-1}\sum_{i=1}^Nf(\tau_i,m_i(\tau))$ stays around its conditional mean $N^{-1}\sum_{i=1}^N\la f(\tau_i,m_i(\tau))|\tau_{l},\forall l\neq i\ra.$ As an immediate consequence, our lemma implies the concentration for the score functions $S(\beta,h|\tau)$ and $Q(\beta,h|\tau)$ (see Corollary~\ref{cor:conc_score} below), which was previously established  in \cite{Chat07} using Stein's method for exchangeable pairs. Our result holds for more general functions and its proof relies on the cavity argument instead.

The proof of Theorem~\ref{thm:J-norm} follows straightforwardly by using the asymptotic results~\cite{BvH16} and the concentration bound~\cite{CDG23} for the operator norm of sparse subgaussian matrices. The proof of Theorem~\ref{thm:TN-check} is the central ingredient of this work.
For simplicity, let us sketch an argument to show that 
\[
\bP\bigl(\min_{\mvgs} T_N(\mvgs) \le c_0\bigr)\le K  \exp (-cN).
\]
We will prove the above by taking a union bound over the spin configuration $\mvgs \in \{-1, 1\}^N$. Notice that 
\begin{equation}\label{eq:inf_representation}
 T_N(\mvgs) = \frac2N \sum_{i\in[N]} (m_i(\mvgs) - \bar{\mvm}(\mvgs))^2 = \frac2N  \inf_{\gamma \in \bR} \norm{\mvm(\mvgs) - \gamma  \mv1}_2^2,
\end{equation}
It is easy to check using Markov inequality that $\|\mvm(\mvgs) \|_\infty \le M:=2^N$ with probability at least $1 - C(1.5)^{-N}$. On this event, 
the infimum over $\gamma$ in \eqref{eq:inf_representation} can be restricted to the interval $[-M, M]$, which can further be approximated by taking minima over an $\epsilon$-net $\mathcal{D}$ of size $O(2^N\eps^{-1})$ for small but constant $\epsilon > 0$. Therefore, it suffices to show that 
uniformly for any $\mvgs \in \{-1, 1\}^N$ and any $\gamma \in \mathcal{D}$, 
\begin{equation}\label{eq:tensor_bd}
\bP\bigl(\norm{\mvm(\mvgs) - \gamma \mv1}_2 \le c_1 \sqrt{N} \bigr) \le K 10^{-N}
\end{equation}
for appropriate positive constant $c_1$.
The event on the left-hand side can immediately be recognized as a small ball probability event. Had the entries of $J$ been independent, the local fields $m_i(\mvgs)$ were independent. In that case, we can use the standard small ball probability result (e.g., Berry-Esseen theorem) for the linear combinations of independent random variables to infer that for each $i$, the probability that $|m_i(\mvgs) - \gamma| < c_1$ is at most $1/10$ for sufficiently small $c_1$ and then tensorizes the bound to obtain \eqref{eq:tensor_bd}. However, since $J$ is symmetric, $m_i(\mvgs)$ are not independent anymore and the tensorization breaks down. 

To overcome this obstacle, we borrow an idea from \cite{Ver14}. We only use the randomness of a principal submatrix of $J$  of linear size while conditioning out the rest of the entries. Let $T \subset [N]$ be the set of the row (and the column) indices belonging to this principal submatrix, which we denote by $J_T$.  Dealing with this smaller matrix instead of $J$ has the critical advantage that the entries of $J_T$  are independent (as long as $|T| \le N/2$). We can then apply the previous small ball probability argument for $J_T \sigma_T$. However, successfully implementing this strategy requires that the matrix $J_T$ is not too sparse in the sense that in the graph $\cG$ the number of edges between each vertex of $T$ and $[N]\setminus T$ tends to $\infty$ as $N \to \infty$.  We apply Erd\"os' probabilistic method to guarantee the existence of such a `good' deterministic set $T$ provided $\cG$ satisfies Assumption~\ref{ass:graph}$(i)$ and $d\to\infty$ as $N\to\infty$. This gives the lower bound for $\min_{\mvgs} T_N(\mvgs)$.

The original term $ \widetilde T_N(\mvgs)$ has an additional complication since $\theta_i(\mvgs) := \sech^2(\gb m_i(\mvgs)+h)$ can be  small if $m_i(\mvgs)$ is big. However, it can be shown that with overwhelming probability, there exists a (random) linear subset of indices of $T$ satisfying  $|m_i(\mvgs)| \le C$, which then yields the lower bound for $ \min_{\mvgs} \widetilde T_N(\mvgs)$ as it follows from the previous step that the local fields $(m_i(\mvgs))_{i \in T}$ have sufficient variability.

\subsection{Structure of the paper.} In the order their statements appear, the proofs for Theorems \ref{thm:exs-cons}, \ref{thm:J-norm}, and \ref{thm:TN-check} are accordingly located in Sections \ref{sec:pf-exist-cons}, \ref{sec:pf-thm-Jnorm}, and \ref{sec:pf-thm-TN}. 

\subsection{Acknowledgements.}
W-KC is partly supported by NSF grants DMS-1752184 and DMS-2246715 and Simons Foundation grant 1027727. He also thanks the hospitality of National Center for Theoretical Sciences in Taiwan during his visit between May 6-10, 2024, where part of this work was completed. AS is partly supported by Simons Foundation MP-TSM-00002716. QW and AS would like to thank Sumit Mukherjee for helpful discussions. 

\section{Proof of Theorem~\ref{thm:exs-cons}}\label{sec:pf-exist-cons}

The proof of Theorem \ref{thm:exs-cons} contains the existence and consistency of the MPLE. The latter is deferred to Section \ref{sec2.3}. For the existence part, as mentioned in the proof sketches, it suffices to show that the pseudo-likelihood function diverges to negative infinity at infinity. We verify this property in Section~\ref{sec2.2} based on a number of lemmas established in the following subsection. We include a flowchart in Figure~\ref{fig:fchart} for the high-level proof strategy of Theorem \ref{thm:exs-cons}.

\begin{figure}
    \centering
\resizebox{0.9\textwidth}{!}{%
\begin{tikzpicture}[
    >=Stealth,
    node distance=2cm,
    box/.style={draw, rounded corners=2pt, thick, minimum width=4.6cm, minimum height=1.3cm, align=center},
    every path/.style={thick}
]

\node[box] (L25) at (-6, 1.8) {Lemma 2.5};
\node[box] (L21) at (-1, 1.8) {Lemma 2.1};

\node[box] (L24) at (6, 3.8) {Lemma 2.4};
\node[box] (P23) at (6, 0.6) {Lemma 2.3};

\node[box] (C22) at (-4, -0.6) {Corollary 2.2};

\node[box] (Texist) at (2, -2.8) {Theorem 1.1: Existence};
\node[box] (Tcons)  at (-4, -2.8) {Theorem 1.1: Consistency};

\draw[->] (L25) -- (C22);
\draw[->] (L21) -- (C22);
\draw[->] (L21) -- (Texist);
\draw[->] (P23) -- (Texist);
\draw[->] (C22) -- (Tcons);

\draw[dashed, ->] (L24) -- (P23) node[midway, right=8pt] {verifying conditions};

\end{tikzpicture}
}
    \caption{Roadmap for the proof of Theorem \ref{thm:exs-cons}. The dashed arrows represent lemma(s) used to validate hypothesis in the target lemma. The solid arrows present proof dependence.}
    \label{fig:fchart}
\end{figure}

\subsection{Auxiliary
 Lemmas} 

 Let $\beta>0$ and $h\in \mathbb{R}$. For $i\in [N],$ let $f_i(x,y)$ be defined on $\{-1,1\}\times\mathbb{R}$ and set
\begin{align*}
    L_i(\mvgs)&=\la f_i(\tau_i,m_i(\tau))|\tau_l=\sigma_l,\forall l\neq i\ra\\
    &=\frac{e^{\beta m_i(\mvgs)+h}f_i(1,m_i(\mvgs))+e^{-\beta m_i(\mvgs)-h}f_i(-1,m_i(\mvgs))}{e^{\beta m_i(\mvgs)+h}+e^{-\beta m_i(\mvgs)-h}}
\end{align*}
for $\mvgs\in\Sigma_N,$
where $\la\cdot\ra$ is the Gibbs expectation with respect to $\tau\sim P_{\beta,h}.$
Our first lemma establishes the following concentration for the linear statistics of $(\tau_i,m_i(\tau))_{i\in [N]}$.

\begin{lem} \label{lem:new-concent}
Let $\mvgt$ be sampled from $P_{\beta, h}$ for some $\beta>0$ and $h\in \mathbb{R}$.
Let $\{f_i\}_{i\in [N]}$ be a collection of functions defined on $\{-1,+1\}\times \mathbb{R}$ with $\max_{i\in [N]}|f_i(x,y)|\leq C(1+|y|)$ for all $(x,y)\in \{-1,+1\}\times\mathbb{R}$ for some positive constant $C$ and $\max_{i\in [N]}(\|\partial_yf_i\|_\infty,\|\partial_{y}^2f_i\|_\infty)<\infty.$
Then there exists a constant $K$ depending only on $\beta$ and $C$ such that for any $N\geq 1,$ we have that 
    \begin{align*}
&\Bigl\la\Bigl(\sum_{i=1}^N\bigl(f_i(\tau_i,m_i(\mvgt))-L_i(\mvgt)\bigr)\Bigr)^2 \Bigr\ra\\
&\leq K(1+\|J\|+\|J\|^2)\Bigl(N+N\|J\|^2+\sum_{i=1}^N\|\partial_y f_i\|_\infty^2+\sum_{i=1}^N\|\partial_y^2f_i\|_\infty^2\Bigr).
\end{align*}
\end{lem}

\begin{proof}
For any $i,j\in[N]$, set 
$$
m_j^{(i)}(\mvgs):=\sum_{k\neq i, j}J_{jk}\sigma_k
$$ 
and let $L^{(i)}_j(\mvgs)$ be the same as $L_j(\mvgs)$ with $m_j(\mvgs)$ replaced by $m_j^{(i)}(\mvgs)$.
Since $L_i(\tau)$ is the expectation of $f_i(\tau_i, m_i(\mvgt))$ conditionally on $(\tau_l)_{l\neq i}$, we have
\begin{align*}
    \bigl\la\bigl(f_i(\tau_i,m_i(\mvgt))-L_i(\mvgt)\bigr)
    \bigl(f_j(\tau_j,m_j^{(i)}(\mvgt))-L_j^{(i)}(\mvgt)\bigr)\bigr\ra=0,
\end{align*}
which implies
\begin{align}
   &  \bigl\la \bigl(f_i(\tau_i, m_i(\mvgt) )-L_i(\mvgt)\bigr)\bigl(f_j(\tau_j, m_j(\mvgt) )-L_j(\mvgt)\bigr)\bigr\ra  \nonumber\\
    & = \bigl\la  X_i(\tau) \big (f_j(\tau_j, m_j(\mvgt))-f_j(\tau_j, m_j^{(i)}(\mvgt))\big ) \bigr\ra \label{eq:ip_decom1}\\
    & + \bigl\la X_i(\tau)c_{ij}(\mvgt)\bigr\ra, \label{eq:ip_decom2}
\end{align}
where
\begin{align*}
c_{ij}(\mvgt)&:=L_j^{(i)}(\mvgt)-L_j(\mvgt),\\
X_i(\gt)&:= f_i(\tau_i,  m_i(\mvgt) )-L_i(\mvgt).
\end{align*}
We bound \eqref{eq:ip_decom2} first.
Note that 
\begin{align}\label{ProofofCor2.2:eq0}
|X_i(\gt)| \le 2 C(1+|m_i(\tau)|).
\end{align}
For $t\in [0,1],$ let $J^{(i,j)}(t)$ be a symmetric square matrix with $J_{ab}^{(i,j)}(t)=J_{ab}$ for all $(a,b)\neq (i,j)$ and $(j,i)$ and $J_{ij}^{(i,j)}(t)=J_{ji}^{(i,j)}(t)=(1-t)J_{ij}.$ Denote $\mathcal{L}_{ij}(t)$ by the function $L_j(\tau)$ with $J$ being replaced by $J^{(i,j)}(t)$. Using Taylor's formula yields that
\begin{align*}
    L_j^{(i)}(\mvgt)&=\mathcal{L}_{ij}(1)\\
    &=\mathcal{L}_{ij}(0)+\mathcal{L}_{ij}'(0)+\frac{1}{2}\int_0^1(1-s)\mathcal{L}_{ij}''(s)ds\\
    &=L_j(\tau)+J_{ij}Y_j+J_{ij}^2Z_j^{(i)},
\end{align*}
where
\begin{align*}
    Y_j&=-\partial_{J_{ij}}L_j(\tau),\\
    Z_j^{(i)}&=\frac{1}{2}\int_0^1(1-s)\partial_{J_{ij}}^2L_j(\tau)\Big|_{J=J^{(i,j)}(s)}ds.
\end{align*}
We emphasize that here $Y_j$ depends only on $j$ since
\begin{align*}
    \partial_{J_{ij}}L_j(\mvgt)&=\frac{e^{\beta m_j(\mvgt)+h}\partial_yf_j(1,m_j(\mvgt))+e^{-\beta m_j(\mvgs)-h}\partial_yf_j(-1,m_j(\mvgt))}{e^{\beta m_j(\mvgt)+h}+e^{-\beta m_j(\mvgt)-h}}\\
    &+\beta\frac{e^{\beta m_j(\mvgt)+h}f_j(1,m_j(\mvgt))-e^{-\beta m_j(\mvgt)-h}f_j(-1,m_j(\mvgt))}{e^{\beta m_j(\mvgt)+h}+e^{-\beta m_j(\mvgt)-h}}\\
    &-\beta\frac{e^{\beta m_j(\mvgt)+h}f_j(1,m_j(\mvgt))+e^{-\beta m_j(\mvgt)-h}f_j(-1,m_j(\mvgt))}{e^{\beta m_j(\mvgt)+h}+e^{-\beta m_j(\mvgt)-h}}\tanh(\beta m_j(\mvgt)+h).
\end{align*}
Now from the given assumption,
\begin{align*}
  |\partial_{J_{ij}}L_j(\mvgt)|&\leq 2\beta C(1+|m_j(\tau)|)+\|\partial_yf_j\|_\infty=:K_j.
\end{align*}
Similarly one can compute the second derivative of $L_j(\mvgt)$ with respect to $J_{ij}$; the exact form is not important here, but it is crucial to realize that
\begin{align*}
    |\partial_{J_{ij}}^2L_j(\mvgt)|&\leq \tilde C\bigl(1+|m_j(\tau)|+\|\partial_yf_j\|_\infty+\|\partial_y^2f_j\|_\infty\bigr)
\end{align*}
for some constant $\tilde C$ depending only on $C$ and $\beta.$ It follows from these bounds that 
\begin{align}\label{ProofofCor2.2:eq1}
    |Y_j|\leq K_j
\end{align} and
\begin{align}
 \nonumber   |Z_j^{(i)}|&\leq \tilde{C}\Bigl(1+\int_0^1(1-s)|(J^{(i,j)}(s)\tau)_j\bigr|ds+\|\partial_yf_j\|_\infty+\|\partial_y^2f_j\|_\infty\Bigr)\\
\label{ProofofCor2.2:eq2}    &\leq \tilde{C}\bigl(1+2\|J\|+\|\partial_yf_j\|_\infty+\|\partial_y^2f_j\|_\infty\bigr)=:\tilde{K}_j,
\end{align}
where we have used the bound $\|J^{(i,j)}(s)\|\leq \|J\|+(1-s)|J_{ij}|\leq 2\|J\|.$
Now, write
\[ \Big| \sum_{i \ne j} X_i(\tau) c_{ij}(\tau) \Big |\le \Big| \sum_{i \ne j} J_{ij} X_i(\tau) Y_j(\tau) \Big| + \sum_{i \ne j} \frac{J_{ij}^2}{2}|X_i(\tau)| |Z_j^{(i)}| .
\]
Note that $J_{ij}=0$ when $i=j$. Using \eqref{ProofofCor2.2:eq0} and \eqref{ProofofCor2.2:eq1} leads to
\begin{align*}
\Bigl|  \sum_{i\neq j} J_{ij} X_i(\gt) Y_j(\tau)  \Bigr|&\leq 2C\|J\|\Bigl(\sum_{i=1}^N(1+|m_i(\tau)|)^2\Bigr)^{1/2}\Bigl(\sum_{j=1}^NK_j^2\Bigr)^{1/2}.
\end{align*}
Similarly, from \eqref{ProofofCor2.2:eq0} and \eqref{ProofofCor2.2:eq2} and using $\|J\circ J\|\leq \|J\|^2,$
\begin{align*}
     \sum_{i \ne j} J_{ij}^2|X_i(\tau)||Z_j^{(i)}|  
    &\leq C\|J\|^2\Bigl(\sum_{i=1}^N(1+|m_i(\tau)|)^2\Bigr)^{1/2}\Bigl(\sum_{i=1}^N\tilde K_i^2\Bigr)^{1/2}.
\end{align*}
From these, \eqref{eq:ip_decom2} is bounded by
\begin{align*}
  \Big| \sum_{i \ne j} X_i(\tau) c_{ij}(\tau) \Big |&\leq 2 C\|J\|\Bigl(\sum_{i=1}^N(1+|m_i(\tau)|)^2\Bigr)^{1/2}\Bigl(\sum_{j=1}^NK_j^2\Bigr)^{1/2}\\
&+C\|J\|^2\Bigl(\sum_{i=1}^N(1+|m_i(\tau)|)^2\Bigr)^{1/2}\Bigl(\sum_{i=1}^N\tilde K_i^2\Bigr)^{1/2}.
\end{align*}

As for \eqref{eq:ip_decom1}, using the Taylor expansion yields
\begin{align*}
    \Bigl|f_i(\tau_j,m_j(\mvgt))-f_i(\tau_j,m_j^{(i)}(\mvgt))+J_{ij}\tau_i\partial_yf_i(\tau_j, m_j(\mvgt))\Bigr|&\leq \frac{1}{2}J_{ij}^2\|\partial^2_{y}f_i\|_\infty.
\end{align*}
Here, we can control similar to the above,
\begin{align*}
   &\Bigl| \sum_{i\neq j}X_i(\tau)\bigl(f_j(\tau_j, m_j(\mvgt))-f_j(\tau_j,m_j^{(i)})\bigr)\Bigr| \\
   &\le \Big| \sum_{i, j} J_{ij} (X_i(\tau) \tau_i ) \big(\partial_yf_j(\tau_j,m_j(\mvgt)) \big) \Big|  + \frac12  \sum_{i, j} J_{ij}^2|X_i(\gt)|\|\partial^2_y f_j\|_{\infty}  \\
   &\leq 2C\|J\|\Bigl(\sum_{i=1}^N(1+|m_i(\tau)|)^2\Bigr)^{1/2}\Bigl(\sum_{i=1}^N\|\partial_yf_i\|_\infty^2\Bigr)^{1/2}\\
&+C\|J\|^2\Bigl(\sum_{i=1}^N(1+|m_i(\tau)|)^2\Bigr)^{1/2}\Bigl(\sum_{i=1}^N\|\partial_y^2f_i\|_\infty^2\Bigr)^{1/2}.
\end{align*}
Putting everything together and noting that $\sum_{i=1}^Nm_i(\tau)^2\leq N\|J\|^2$, we have
\begin{align*}
   &\sum_{i\neq j}\bigl\la \bigl(f_i(\tau_i, m_i(\mvgt) )-L_i(\mvgt)\bigr)\bigl(f_i(\tau_j, m_j(\mvgt) )-L_j(\mvgt)\bigr)\bigr\ra\\
   &\leq K(\|J\|+\|J\|^2)\Bigl(N+N\|J\|^2+\sum_{i=1}^N\|\partial_y f_i\|_\infty^2+\sum_{i=1}^N\|\partial_y^2f_i\|_\infty^2\Bigr)
\end{align*}
for some constant $K$ depending only on $\beta$ and $C.$ 
Finally, for the sum over $i=j$, it can be easily bounded by $C''N(+\|J\|^2)$ for some constant $K'$ depending only on $C$ and this completes our proof.

\end{proof}

As an immediate corollary of this lemma with $f(x,y)=xy$ and $f(x,y)=x$, we have 

\begin{cor}[Second moment bounds for score functions]\label{cor:conc_score}
Assume that $\|J \| = O(1).$ Then there exists a constant $C$ such that 
\[
\la (S(\gb,h|\mvgt))^2 \ra \le CN, \ \  \la (Q(\gb,h|\mvgt))^2 \ra \le C N,
\]
where $S$ and $Q$ are defined in \eqref{eq:score} and \eqref{eq:score2}.
\end{cor}

Now we start with the proof of Theorem~\ref{thm:exs-cons} existence part. 
\subsection{Proof of Theorem~\ref{thm:exs-cons}:Existence of the MPLE}\label{sec2.2}
We begin with the following lemma, which identifies sufficient conditions for the existence of MPLE.

\begin{lem}\label{fix:prop}
	Let $\gs \in \{-1, +1\}^N$ and assume that there exist distinct $1 \le i, j, k, \ell \le N$ and $a \in \bR$ such that
    \begin{align} \label{suff_cond}
    \begin{split}
        & \gs_i = 1, m_i(\mvgs) >a, \,\, \gs_j = -1,  m_j(\mvgs) > a, \\
        & \gs_k = 1,  m_k(\mvgs) < a, \,\,\gs_l = -1,  m_l(\mvgs) < a.
        \end{split}
    \end{align}
	then  $\lim_{ |\beta| + |h| \to \infty} \bar L(\beta, h| \sigma)  =  - \infty.$
\end{lem}

\begin{proof}
For brevity of notation, we denote $m_s(\sigma)$ by $m_s$. Noting that   \begin{align*}
\bar{L}(\beta, h| \mvgs)  &= \frac1N\sum_{s=1}^N  \Big( \sigma_s( \beta m_s +h)  - \log \cosh ( \beta m_s +h) -\log 2\Big) \\
  &\le \frac1N \sum_{s=1}^N  \Big( | \beta m_s +h|  - \log \cosh ( \beta m_s +h) - \log 2\Big)
\end{align*}  
and each term in the above sum is non-positive, it is enough to show that given $\mvgs$, for any $K>0$, there exists some $C>0$ such that the following statement holds: for any $(\beta,h)\in \mathbb{R}^2$ with $|\beta|+|h|\geq C,$ there exists some $s=s(\beta,h)\in [N]$ such that 
\begin{equation}\label{cond_each_s_diverge}
   \sigma_s( \beta m_s +h)  \le -K. 
\end{equation}
We will argue by contradiction. Assume that $i, j, k,$ and $l$ satisfy~\eqref{suff_cond} and each of them violates \eqref{cond_each_s_diverge}.  Then we must have
\begin{align*}
    \sigma_i = 1, \quad \beta m_i + h >  - K \\
    \sigma_j = -1, \quad \beta m_j + h <  K \\
    \sigma_k = 1, \quad \beta m_k + h >-  K \\
    \sigma_l = -1, \quad \beta m_l + h <  K,
\end{align*}
This confines $h$ to the interval:
\begin{equation}\label{interval_h}
 \max(-\beta m_i - K, -\beta m_k - K) < h < \min(-\beta m_j + K, -\beta m_l + K).   
\end{equation}

For this interval to be non-empty, its upper bound must be greater than its lower bound.

 \textbf{Case $\beta \ge 0$}: The conditions $m_j > a > m_k$ imply $m_j > m_k$. The non-empty interval constraint implies $-\beta m_k - K < -\beta m_j + K$, which simplifies to $\beta(m_j - m_k) < 2K$. Therefore, we have 
    \begin{equation}\label{positive_beta}
     \beta < 2K/(m_j - m_k).
    \end{equation}
    
 \textbf{Case $\beta < 0$}:  The conditions $m_i > a > m_l$ imply $m_i > m_l$. The constraint implies $-\beta m_i - K < -\beta m_l + K$, which simplifies to $-\beta(m_i - m_l) < 2K$. 
        Therefore, we have 
            \begin{equation}\label{negative_beta}
 - \beta < 2K/(m_i - m_l).
 \end{equation}
Combining \eqref{positive_beta} and \eqref{negative_beta}, we obtain
\[ |\beta| < \max \left (\frac{2K}{m_j - m_k}, \frac{2K}{m_i - m_l}  \right),  \]
which, in combination with \eqref{interval_h} forces $|h|$ to be bounded. So, $|\beta| + |h|$ can not be arbitrarily large. This concludes the proof of the lemma.
\end{proof}

The following lemma, which is purely deterministic, states a key step to verify the conditions in Lemma~\ref{fix:prop}. Let \[
\tilde{T}_N := N^{-2} \sum_{i,j=1}^N \sech^2(\gb_0m_i +h_0)\sech^2(\gb_0m_j +h_0)(m_i-m_j)^2.
\]
Note that this is a deterministic object and different from \eqref{eq:det-hess}.
\begin{lem}\label{fix:lem1}
    Let $\beta_0>0,h_0 \in \mathbb{R},$ and $c>0$. There exist positive constants $\delta,\epsilon,K$,  depending only on $\beta_0, h_0$ and $c$, such that for any sequence $m_1, m_2, \ldots, m_N$ of real numbers satisfying $\tilde{T}_N \ge c$,
    there exists an integer $r$ with $|r| \le K/ \delta$ and $
    -K \le (r-1) \delta < r \delta \le K$ such that 
    \begin{align}\label{fix:eq3} |\{i \in [N]: m_i \in  [-K,  (r-1)\delta]  \}|  \ge \eps N \text{ and } |\{i \in [N]: m_i \in  [  r \delta, K ]  \}|  \ge \eps N.
    	\end{align}
\end{lem}
\begin{proof}
    Since $\sech^2(\beta_0 x + h_0) \le \min(1, 4 e^{-2 (\beta_0 |x| - |h_0|)})$, we can choose $K > 0$ sufficiently large so that
    \begin{equation} \label{eq:restricted_var}
        \sum_{i,j: |m_i|, |m_j| \le K } (m_i - m_j)^2 \ge \frac{c}{2} N^2,
    \end{equation}
    which implies that at least $\eta N$ of the $m_i$'s lie within $[-K, K]$, where $\eta = \sqrt{c/2} / (2K)$. Set
    \[
    \eps = \min \left( \frac{\eta}{4}, \frac{c}{80 K^2} \right)\,\, \text{and} \,\, \delta = \min \left( \frac{\sqrt{c}}{5}, \frac{K}{2} \right)
    \]
    and define
    \[
    r = \max \left\{ s \in \mathbb{Z} : \left| \left\{ i \in [N] : m_i \in [s \delta, K] \right\} \right| \ge \eps N \right\}.
    \]
    Suppose, for contradiction, that
    \[
    \left| \left\{ i \in [N] : m_i \in [-K, (r - 1)\delta] \right\} \right| < \eps N.
    \]
    This implies that at most $2\eps N$ of the $m_i$ can lie in the interval $[-K, K] \setminus [(r-1)\delta, (r+1)\delta]$. We now derive an upper bound on the left-hand side of \eqref{eq:restricted_var}.

    \medskip
    \noindent
  $(I)$ The total contribution from pairs $i, j$ such that at least one of $m_i$ or $m_j$ lies outside $[(r-1)\delta, (r+1)\delta]$ is at most
    \[
    2 \cdot N \cdot 2\eps N \cdot 4K^2 \le \frac{c}{5} N^2.
    \]

    \medskip
    \noindent
    $(II)$ The total contribution from pairs $i, j$ such that both $m_i$ and $m_j$ lie within $[(r-1)\delta, (r+1)\delta]$ is at most
    \[
    N^2 \cdot (2\delta)^2 \le \frac{c}{5} N^2.
    \]

    \medskip
    \noindent
    By the choice of $\eps$ and $\delta$, the sum of these contributions is at most $2c N^2/5$, which contradicts \eqref{eq:restricted_var}. Therefore,
    \[
    \left| \left\{ i \in [N] : m_i \le (r - 1)\delta \right\} \right| \ge \eps N
    \]
    and the lemma follows.
\end{proof}

Now we complete the proof of the existence part of Theorem~\ref{thm:exs-cons}.
\begin{proof}
For $\ell, u\in \bR$ with $\ell<u$ let
	\[ S_{\ell, u} :=  \left\{ i \in [N] : m_i \in [\ell, u] \right\}.\]
	First of all, we claim that for any $a\in \mathbb{R}$ and $\epsilon,\delta>0,$ the following events have negligible probability as $N \to \infty,$ 
\begin{align}\label{fix:proof:eq1}
|S_{-K, a}| \ge \eps N\,\,\text{and}\,\,\mbox{the spins in}\,\,\{ \sigma_i : i \in S_{-\infty, a+ \delta/3} \} \text{ have all the same signs}
\end{align}
and
\begin{align}
	\label{fix:proof:eq2}
	|S_{a+ \delta, K}| \ge \eps N\,\,\mbox{and}\,\,\mbox{the spins in}\,\,\{ \sigma_i : i \in S_{a+ 2 \delta/3, \infty} \} \text{ have all the same signs}.
\end{align}
We will only prove \eqref{fix:proof:eq2} since \eqref{fix:proof:eq1} can be treated similarly. Suppose $|S_{a +\delta, K}| \ge \eps N$ and $\sigma_i =1 $ for all $i \in S_{a + 2 \delta/3, \infty}$. Choose a smooth function $\rho : \mathbb R \to [0, 1]$ such that 
$\rho \equiv 1$ on $[a+\delta, \infty)$ and $\rho$ vanishes outside $[a+2\delta/3, \infty).$ Take $f(\sigma_i, m_i) = \sigma_i \rho(m_i).$ 
Note that  for any $i\in [N],$
\begin{align*}
f(\sigma_i,m_i)-	\la f(\sigma,m_i)\ra_{-i}&=\rho(m_i)\bigl(\sigma_i-\tanh(\beta m_i+h)\bigr).
\end{align*}
This is nonnegative if $i\in S_{a+2\delta/3,\infty}$ and is equal to zero if $i\notin S_{a+2\delta/2,\infty}.$ On the other hand, there exists a constant $\kappa>0$ such that 
\[ f(\sigma_i, m_i)  - \langle f(\sigma_i, m_i) \rangle_{-i} \ge \kappa \text{ for } i \in  S_{a +\delta, K}. \]
Together with the assumption $|S_{a +\delta, K}| \ge \eps N$, we have
\[ \sum_{i=1}^N \big( f(\sigma_i, m_i)  - \langle f(\sigma_i, m_i) \rangle_{-i} \big) \ge \kappa \eps N,\]
contradicting the concentration results established in Lemma~\ref{lem:new-concent}. Hence, \eqref{fix:proof:eq2} must be negligible and this completes the proof of our claim.
Next, let $\eta\in (0,1)$ be fixed. From the positivity assumption of $\widetilde T_N$ and Lemma \ref{fix:lem1},  there exist positive constants, $\delta,\epsilon,K$ depending on $\beta_0,h_0,$ and $c$ such that for any $N$ large enough, with probability at least $1-\eta$, there exists some $a\in \mathbb{R},$
\begin{align*}
|S_{-K,a}|\geq \epsilon N\,\,\mbox{and}\,\,|S_{a+\delta,K}|\geq \epsilon N
\end{align*}
Noting that there are only finitely many possible $a$'s; indeed, $a=(r-1)\delta$ for some integer $r\in \mathbb{Z}$ with $|r|\leq K/\delta$. It follows from the above claim and an union bound argument that as long as $N$ is large enough, with probability at least $1-2\eta,$ the condition~\eqref{suff_cond} in Lemma \ref{fix:prop} hold and thus, on this event, $\lim_{ |\beta| + |h| \to \infty} \bar L(\beta, h| \sigma)  =  - \infty.$

\end{proof}

\subsection{Proof of Theorem~\ref{thm:exs-cons}: Consistency of MPLE}\label{sec2.3}

We begin with a simple lemma.

\begin{lem} \label{lem:T_n_restricted_lb}
    Fix $\beta > 0$ and $h \in \mathbb{R}$. Let $\mvgs \in \{-1, +1\}^N$.
    Suppose that $\| J \|  = O(1)$ and 
    $\widetilde T_N(\mvgs) \ge c.$ Then there exists a constant $M>0$, which depends only on $c, \beta, h,$ and $\|J\|$ such that 
    \[ \frac{1}{N^2} \sum_{ i, j \in I_M(\mvgs)} (m_i(\mvgs) - m_j(\mvgs))^2  \ge \frac{c}{2}, \]
    where $I_M(\mvgs)=\{i\in [N]:|m_i(\mvgs)|\leq M\}.$
\end{lem}

\begin{proof}
    Let $M$ be sufficiently large such that $\beta M > |h|$. We have 
    \begin{align*}
      \frac{1}{N^2} \sum_{ i, j \in I_M(\mvgs)} (m_i(\mvgs) - m_j(\mvgs))^2  &\ge 
      \frac{1}{N^2} \sum_{ i, j \in I_M(\mvgs)} \theta_i(\mvgs) \theta_j(\mvgs) (m_i(\mvgs) - m_j(\mvgs))^2 \\
      &\ge c - \frac{1}{N^2} \sum_{ \{i, j\} \cap I_M^c(\mvgs) \ne \emptyset} \theta_i(\mvgs) \theta_j(\mvgs) (m_i(\mvgs) - m_j(\mvgs))^2 \\
      &\ge c - 
      \frac{\sech^2(\beta M  - |h|)}{N^2} \sum_{ i, j}  (m_i(\mvgs) - m_j(\mvgs))^2.
    \end{align*}
    On the other hand, 
\begin{align*}
    \frac{1}{N^2} \sum_{ i, j}  (m_i(\mvgs) - m_j(\mvgs))^2 \le  \frac{2}{N} \sum_{ i}  m_i(\mvgs)^2 \le 2\|J\|^2.
\end{align*}    
Combining these two estimates, the lemma follows that if we choose $M$ large enough to satisfy
$2\sech^2(\beta M  - |h|)\|J\|^2 \le c/2 $.

\end{proof}

The proof for the consistency of MPLE in Theorem \ref{thm:exs-cons} proceeds as follows.
For notational clarity, we will make the dependence of the quantities $\theta_i$ and  $\widetilde T_N$ on parameters $\beta$ and $h$ explicit throughout the proof. 
Let $\mvgt \sim P_{\beta_0, h_0}$. For $(\beta, h ) \in \mathbb{R}^2$, let $\gl_1(\gb,h |\mvgt)\geq 0$ be the minimum eigenvalue of the negative Hessian matrix, $H(\gb,h|\mvgt)$. 
Note that
\begin{align} \label{eq:mineig_bd1}
\gl_1(\gb, h |\mvgt) & \ge \frac{\abs{H(\gb ,h |\mvgt)}}{\mathrm{Tr}(H(\gb,h|\mvgt))} =  \frac{\frac12 \sum_{i,j=1}^N \theta_{i}(\beta, h| \mvgt) \theta_{j}(\beta, h| \mvgt)  (m_i(\mvgt)-m_j(\mvgt))^2}{\mathrm{Tr}(H(\gb,h|\mvgt))},
\end{align}
where $\theta_{i}(\beta, h| \mvgt) =\sech^2(\beta m_i(\mvgt) + h)$.
To bound the right hand side of \eqref{eq:mineig_bd1} from below, we note
\begin{align*}
 &\sum_{i,j=1}^N \theta_{i}(\beta, h| \mvgt) \theta_{j}(\beta, h| \mvgt)  (m_i(\mvgt)-m_j(\mvgt))^2 \\
 &\ge \sech^4( \beta M + |h|)   \sum_{i, j \in I_M(\mvgt)} (m_i(\mvgt)-m_j(\mvgt))^2
\end{align*}
for any $M>0$.  Also, note that 
since $\theta_i(\beta,h|\mvgt)\leq 1$ and $\sum_{i=1}^Nm_i^2(\mvgt)\leq \|J\|^2N,$
\[
\mathrm{Tr}(H(\gb,h |\mvgt)) = \sum_{i=1}^N\theta_i(\gb ,h|\mvgt) (m_i^2(\mvgt)+1)\le N(1+C^2).
\]
Therefore, we obtain 
\begin{align}\label{eq:mineig_bd2}
    \gl_1(\gb, h |\mvgt) \ge \frac{\sech^4( \beta M + |h|)}{2(1+ C^2) N}   \sum_{i, j \in I_M(\mvgt)} (m_i(\mvgt)-m_j(\mvgt))^2.
\end{align}
This yields the following locally uniform lower bound for $\gl_1(\gb, h |\mvgt)$ that for any $M>0$, there exist some positive $r,\eta>0$ such that
\begin{align}\label{eq:mineig_uni_bd}
\inf_{\norm{\gb-\gb_0,h-h_0}\le r} \gl_1(\gb, h|\mvgt) \ge \frac{\eta }{N} \sum_{i, j \in I_M(\mvgt)} (m_i(\mvgt)-m_j(\mvgt))^2.
\end{align}

Next, we construct an interpolation between the pseudolikelihood estimator and the ground truth by letting
\begin{align*}
    \gb(t)= t \hat{\gb} + (1-t)   \gb_0\,\,\mbox{and}\,\, h(t)= t  \hat{h} + (1-t)   h_0
\end{align*}
for $t\in [0,1].$
Set $f:[0,1]\to \bR$ as
\begin{align*}
 f(t)= (\hat{\gb} - \gb_0)  S(\gb(t), h(t)|\mvgt) + (\hat{h} - h_0) Q(\gb(t), h(t)|\mvgt) .
\end{align*}
Since $S(\hat{\gb},\hat{h}|\mvgt)=0$ and $Q(\hat{\gb},\hat{h}|\mvgt)=0$, it is clear
\begin{align*}
    \abs{f(1)- f(0)} = \bigl|(\hat{\gb} - \gb_0) S(\gb_0, h_0|\mvgt) + (\hat{h} - h_0) Q(\gb_0, h_0|\mvgt) \bigr|.
\end{align*}
By Corollary~\ref{cor:conc_score}, we have 
\[
S(\gb_0,h_0|\mvgt), Q(\gb_0,h_0|\mvgt) = O_p(\sqrt{N}).
\]
Applying the Cauchy-Schwarz inequality, we deduce 
\begin{align}\label{eq:asy-YN}
\abs{f(1) - f(0)} = O_p(\sqrt{N} Y_N),
\end{align}
where $Y_N:=\|(\hat{\gb}-\gb_0, \hat{h}-h_0 )\|_{2}$. 

Now to bound $Y_N$, we need to control $\abs{f(1) - f(0)}$.
For any $t\in[0,1]$, the derivative of $f$ is bounded from below by
\begin{align*}
    f'(t) &= (\hat{\gb}-\gb_0, \hat{h}-h_0)  H(\gb(t),h(t)|\mvgt)   (\hat{\gb}-\gb_0, \hat{h}-h_0)^T \\
    &\geq \gl_1(\gb(t),h(t)|\mvgt) Y_N^2.
\end{align*}
From the fact $\| (\beta(t) - \beta_0,  h(t) - h_0)\|_2  = t Y_N $ and 
the uniform lower bound \eqref{eq:mineig_uni_bd}, we obtain 
\begin{align}\label{eq:lem1.2-2}
    f(1)-f(0) &= \int_0^1 f'(t) dt \ge \int_{0}^{\min \big( 1, \frac{r}{Y_N}\big) } f'(t) dt \nonumber \\
    &\ge \min \big( 1, \frac{r}{Y_N}\big)    \frac{\eta Y_N^2}{N} \sum_{i, j \in I_M(\mvgt)} (m_i(\mvgt)-m_j(\mvgt))^2.
\end{align}
From the assumption (iii) of Theorem~\ref{thm:exs-cons}, it follows that given $\eps >0$, there exists $c>0$ such that with $P_{\beta_0, h_0}$-probability at least $1- \eps/2$, we have $\widetilde{T}_N(\gb_0,h_0|\mvgt) \ge c$.
Furthermore, thanks to Lemma~\ref{lem:T_n_restricted_lb}, we can find  $M$ such that on the same event,
\[ \frac{1}{N^2} \sum_{i, j \in I_M(\mvgt)} (m_i(\mvgt)-m_j(\mvgt))^2 \ge \frac{c}{2}.\]
On the other hand, \eqref{eq:asy-YN} implies that there exists $C$ such that 
\[ | f(1)  -  f(0)| \le C \sqrt{N} Y_N\]
with $P_{\beta_0, h_0}$-probability at least $1- \eps/2$. Since both of the conditions hold simultaneously with probability at least $1-\eps$,
it now follows from \eqref{eq:lem1.2-2} that with probability at least $1-\eps$,
\begin{align*}
\min(Y_N, r) \le \frac{2C}{\eta c\sqrt{N}}.
\end{align*}
Since $r$ is independent of $N$, it follows that
$Y_N  = O_p(N^{-1/2})$ as claimed.

\section{Proof of Theorem~\ref{thm:J-norm}}\label{sec:pf-thm-Jnorm}
We first establish the bound for the expected operator norm of $J = \frac{1}{\sqrt{d}} A \odot G$. It relies on the following fundamental result in Corollary 3.3 of~\cite{BvH16}, which we recall as follows.

\begin{lem}[Corollary 3.3~\cite{BvH16}]
    Consider the $N\times N$ symmetric random matrix $X$ whose entries $X_{ij} :=b_{ij} g_{ij}$ for $i\ge j$, where $b_{ij} \in \bR$ and $g_{ij}$ are independent centered subgaussian. Then 
    \[
    \E \norm{X} \lesssim \max_i \sqrt{\sum_{j} b_{ij}^2} + \sqrt{\log N} \cdot \max_{ij} \abs{b_{ij}}.
    \]
\end{lem}
Appling this result in our setting for the matrix $J$, it immediately gives the following lemma.
\begin{lem}\label{lem:mean-oper-norm}
    If the entries of the symmetric matrix $G$ are independent centered subgaussian random variables, then we have for a universal constant $C>0$ such that
    \[
    \E \norm{J} \le C \left( \sqrt{\frac{\log N }{d}} + \sqrt{\frac{d_{\rm{max}}}{d}}\right).
    \]
\end{lem}
\begin{proof}
    The proof is immediate after applying the Corollary 3.3 in~\cite{BvH16}.
\end{proof}
Now it can be seen that under the Assumption~\ref{ass:graph}, we have $\E \norm{J} = O(1)$. To complete the proof of Theorem~\ref{thm:J-norm}, the rest is to prove the following concentration results for $\norm{J}$.

\begin{lem}\label{lem:Jnorm-conc}
     If the entries of the symmetric matrix $G$ are independent centered $\kappa^2$-subgaussian random variables for some $K\ge 1$, then we have for some universal constant $C>0$,
     \[
     \bP(\abs{\norm{J} - \E \norm{J}}\ge t) \le  2\exp\left(-\frac{Cdt^2}{\kappa^4 \log(Nd)}\right).
     \]
\end{lem}

The major ingredient for the proof of Lemma~\ref{lem:Jnorm-conc} is the following concentration result for subgaussian matrices appeared in~\cite{CDG23}. We include this result for reader's convenience.

\begin{lem}[Proposition A.1~\cite{CDG23}]
Let $\mvX:=(X_1,\ldots, X_d)$ be a vector of independent $K$-subgaussian random variables for some $K\ge 1$. Let $f:\bR^d \to \bR$ be a convex 1-Lipschitz function. Then for $s\ge 0$,
\[
\bP(\abs{f(\mvX) - \E f(\mvX)}\ge s) \le 2 \exp\left(-\frac{cs^2}{K^2 \log d}\right),
\]
where $c>0$ is some universal constant. Moreover, $\E f(\mvX)$ can be replaced by any median of $f(\mvX)$. 
\end{lem}

\begin{proof}[Proof of Lemma~\ref{lem:Jnorm-conc}]
    We will apply the Proposition A.1 in~\cite{CDG23} for subgaussian random matrix to complete the proof. Note that in our setting, the operator norm is obviously a convex function. On the other hand, since for general matrices $A,B$,
    \[
    \abs{\norm{A} - \norm{B}} \le \norm{A-B}_F,
    \]
    where $\norm{\cdot}_F$ denotes the Frobenius norm. This implies the Lipschitz property of the operator norm. There are in total $Nd$ many nonzero entries in $J$, then directly from Proposition A.1 of~\cite{CDG23},
    \[
    \bP(\abs{\norm{J} - \E \norm{J}}\ge t) \le 2\exp\left(-\frac{Cdt^2}{\kappa^4 \log(Nd)}\right).
    \]
\end{proof}
Now we put the above two lemmas together to complete the proof of Theorem~\ref{thm:J-norm}.
\begin{proof}[Proof of Theorem~\ref{thm:J-norm}]
    Using the Assumption~\ref{ass:graph}, it is easy to see that $\E \norm{J} = O(1)$ from Lemma~\ref{lem:mean-oper-norm}. On the other hand, since $d \gg \log N$, the r.h.s in the concentration bound of Lemma~\ref{lem:Jnorm-conc} can be bounded as follows, 
    \[
    2\exp\left(-\frac{Cdt^2}{K^2 \log(Nd)}\right) \le 2\exp\left(-\frac{C'dt^2}{K^2 \log(N)}\right).
    \]
    Using $d\gg \log N$ in the Assumption~\ref{ass:graph}, we have the desired results. 
\end{proof}

\section{Proof of Theorem~\ref{thm:TN-check}}\label{sec:pf-thm-TN}

The proof of Theorem~\ref{thm:TN-check} uses a linear size set of the vertices of $\cG$ with all of its vertices having at least $c d$ many neighbors outside the set. We start by showing that such a set always exists.

For a subset of vertices $Q \subseteq [N]$, denote by $\dout_i(Q)$  the number of neighbors of $i$ outside~$Q$ in $\cG$, i.e., 
\[ \dout_i(Q) = | \{ j \in Q^c: (i, j) \in E  \}|.\]
\begin{defn}
   For $\delta> 0, k \ge 1$,  we call a subset of vertices $T \subseteq [N]$ of $\cG$ $(\delta, k)$-good if 
    $|T| \ge \delta N$ and $\dout_i(T) \ge k$ for all $i\in T$.
\end{defn}  
The following lemma guarantees the existence of a \emph{good} set $T$ under the assumption of $d_{\max} \le Cd$.
\begin{lem}\label{lem:T-set}
Let $\cG$ be any graph on $N \ge 2$ vertices such that  $d_{\max} \le Cd$. There exists a subset $T \subseteq [N]$ that is $(1/(16C), d/4)$-good.
\end{lem}


To establish this lemma, we need an elementary result that states that given a graph $\cG$, we can always find a good expander subset $S \subseteq [N]$ with a large number of outgoing edges, see, e.g. Theorem 2.2.2 of~\cite{AS08}.
\begin{lem}\label{lem:bipar}
     Let $\cG=([N],E)$ be a graph with $N \ge 2$ vertices. Then there exists $ S\subseteq [N]$ with $|S| = \lfloor {N}/{2} \rfloor$ and $\abs{E(S,S^c)}\ge \abs{E}/{2} = {Nd}/{4}$, where
     $E(S,S^c) = \{ (i ,j) \in E: i \in S, j \in S^c \}$ and  $d$ is the total average degree.
 \end{lem}
\begin{proof}
Let $\ell = \lfloor {N}/{2} \rfloor$. Choose a random set $U \subseteq [N]$ with $|U| = \ell$ uniformly at random from the set of all subsets of $[N]$ of size $\ell$. 
Note that for each $e\in E,$
\[
\bP(e\in E(U, U^c))  = 1 - \frac{{\binom{\ell}{2}}+ {\binom{N-\ell}{2}}}{{\binom{N}{2}}} \ge \frac{1}{2}.
\]
Since we can write 
\[ | E(U, U^c)|  = \sum_{e \in E} \ind_{\{  e \in E(U, U^c)\} }, \]
it follows from the linearity of expectation that  $\bE | E(U, U^c)|  \ge |E|/2. $  Therefore, there exists a (deterministic) subset $S$ of size $\ell$ such that $| E(S, S^c)| \ge |E|/2, $ as claimed.
\end{proof}

\begin{proof}[Proof of Lemma~\ref{lem:T-set}]
Fix a subset $S$ as guaranteed by the previous lemma. 
Let
\begin{align*}
   T=\Bigl\{i \in S: \dout_i(S) \ge \frac{d}{4}\Bigr\},
\end{align*}
which contains the vertices in $S$ with at least ${d}/{4}$ many outgoing edges to $S^c$. Now let $I$ be a uniform random variable on $S$. Write 
\begin{align*}
\frac{1}{\abs{S}}\sum_{i\in S} \dout_i(S) = \E \dout_I(S) & \le \max_{i\in S}\dout_i(S) \cdot \bP(I\in T) + \frac{d}{4} \bP(I \notin T)  \\
& \le Cd  \bP(I \in T) + \frac{d}{4} ( 1- \bP(I \in T)).
\end{align*}
Note that since $$ \frac{1}{\abs{S}}\sum_{i\in S} \dout_i(S)=\frac{E(S,S^c)}{|S|}\geq \frac{Nd/4}{\lfloor N/2\rfloor} \ge \frac{d}{2},$$ we have
\[
\frac12 \le (C-1/4) \bP(I \in T) +  \frac{1}{4}.
\]
It follows that
\[
\bP(I \in T) \ge \frac{1/4}{C- 1/4} \ge \frac{1}{4C},
\]
which implies $\abs{T} \ge {\abs{S}}/{(4C)} = {N}/{(16C)}$. Furthermore, for each $i \in T$, we have $\dout_i(T) \ge \dout_i(S) \ge d/4$. This completes the proof of Lemma~\ref{lem:T-set}.
\end{proof}
From now on, for the rest of Section~\ref{sec:pf-thm-TN}, we fix a $(1/(16C), d/4)$-good subset $T$ as ensured by Lemma~\ref{lem:T-set}. We employ a union bound over $\mvgs \in \{-1,+1\}^N$ to prove  Theorem~\ref{thm:TN-check}; the heart of the proof lies the following crucial uniform bound on the $\ell_2$-deviation of $\mvm(\mvgs)$ from a constant vector.

\begin{prop}\label{prop:TN-reduced}
Under the same conditions of Theorem~\ref{thm:TN-check}, there exist positive constants $\gd$ and $K$ such that for any   $N\geq 1$,
$\mvgs \in \{-1,+1\}^N$, and  $\gamma \in \bR$, we have
\begin{align}\label{eq:res-no-theta}
\bP\Bigl( \min_{ A\subseteq T: |A| \ge \frac{\abs{T}}{2}}  \norm{\mvm_A(\mvgs) - \gamma\mv1_A}_{2} \le \gd \sqrt{N}\Bigr) \le \frac{K}{10^N},
\end{align}
where $m_A(\mvgs):=(m_i(\mvgs))_{i\in A}$ and $\mv1_A:=(1)_{i\in A}$.
\end{prop}

We
first see how this proposition implies that $T_N(\mvgs)$ is bounded away from zero uniformly in $\mvgs$ with high probability, which is a weaker form of Theorem~\ref{thm:TN-check}. The proof of Proposition~\ref{prop:TN-reduced} is deferred to Subsection~\ref{ssec:pf-TN}.

\begin{prop}\label{prop:TN}
   Under the same conditions of Theorem~\ref{thm:TN-check}, there exist positive constants $c_1,c$, and $K$ such that for any $N\geq 1,$
    \[
    \bP\bigl(\min_{\mvgs}T_N(\mvgs) \le c_1\bigr) \le K \exp( - cN).
    \]
\end{prop}

\begin{proof}
First, for any $M>0,$ if $
\max_{\mvgs}|\bar m(\mvgs)|\leq M,$ 
we can write
\begin{align*}
\frac1N\min_{\mvgs}  \sum_{i \in [N]} (m_i(\mvgs) - \bar{m}(\mvgs))^2& = \frac1N\min_{\mvgs}  \norm{\mvm(\mvgs) - \bar{m}(\mvgs) \mv1}_2^2\\
&= \frac{1}{N} \min_{\mvgs} \inf_{\gamma\in [-M,M]}\norm{\mvm(\mvgs) - \gamma \mv1}_2^2,
\end{align*}
which implies that
\begin{align*}
     \bP\Bigl(\frac{1}{N}\min_{\mvgs} \inf_{\gamma\in \bR} \bigl\|\mvm(\mvgs) - \gamma\mv1\bigr\|_2^2 \le c_1\Bigr) 
    &\le  \bP\left(\max_{\mvgs} \abs{\bar{m}(\mvgs)} > M\right) \\
   & + \bP\Bigl( \frac{1}{N} \min_{\mvgs} \inf_{\gamma\in [-M, M]}\norm{\mvm(\mvgs) - \gamma  \mv1}_2^2 \le c_1\Bigr).
\end{align*}
To bound the two probabilities on the right-hand side, first of all, note that
\[
\abs{\bar{m}(\mvgs)} = \Bigl|\frac1N\sum_{i\in[N]} m_i(\mvgs)\Bigr|\le \frac1N \sum_{i,j\in[N]} \abs{J_{ij}},
\]
 by Markov's inequality, for a large constant $M=2^N$,
\begin{align*}
\bP\left(\max_{\mvgs} \abs{\bar{m}(\mvgs)} > M\right) &\le \frac{\sum_{i,j\in[N]} \E \abs{J_{ij}}}{NM} \\
&= \frac{d  \E \abs{J}}{2^N} \le \frac{C}{1.5^N}.
\end{align*}
As for the second probability, consider an arbitrary partition $\mathcal{D}=(e_i)_{i\in [2M]}$ of $[-M,M]$ with $|e_i-e_{i-1}|<\eps$ for $i\in [2M]$.
From this, 
\[
\frac1N \min_{\mvgs} \inf_{\gamma \in[-M,M]} \norm{\mvm(\mvgs) - \gamma  \mv1}_2^2 \ge \frac1N \min_{\mvgs} \inf_{\gamma \in \cD} \norm{\mvm(\mvgs) - \gamma  \mv1}_2^2 - \eps^2.
\]
and thus,
\[
\bP\Bigl(\frac1N \min_{\mvgs  } \inf_{\gamma\in [-M,M]} \norm{\mvm(\mvgs) - \gamma  \mv1}_2\le c_1 \Bigr) \le \bP\Bigl(\frac1N \min_{\mvgs } \inf_{\gamma\in \cD} \norm{\mvm(\mvgs) - \gamma  \mv1}_2\le c_1 + \eps^2 \Bigr).
\]
Applying Proposition~\ref{prop:TN-reduced} with $\eps^2+c_1= \gd^2$, we have that for any $\mvgs$ and $\gamma \in \cD$,

\[
\bP(\norm{\mvm(\mvgs)-\gamma  \mv1}_2 \le \gd \sqrt{N}) \le \bP\Bigl( \min_{ A\subseteq T: |A| \ge \frac{\abs{T}}{2}}  \norm{\mvm_A(\mvgs) - \gamma \mv1_A}_{2,A} \le \gd \sqrt{N}\Bigr) \le \frac{K}{10^N}.
\]
Finally taking union bound over $\mvgs$ and $\gamma$ yields
\[
\bP\Bigl(\min_{\mvgs} \inf_{\gamma\in \cD} \norm{\mvm(\mvgs) - \gamma \mv1}_2\le \gd \sqrt{N}\Bigr) \le c' \cdot 10^{-N}\cdot 2^{N} \cdot M,
\]
completing our proof.
\end{proof}



\subsection{Proof of Proposition~\ref{prop:TN-reduced}}\label{ssec:pf-TN}
The proof of Proposition~\ref{prop:TN-reduced} relies on certain small-ball probability estimates. To this end, we first recall the following classical small-ball probability result for the linear combination of independent random variables. 

\begin{lem}[Corollary 2.9~\cite{RV08}]\label{thm:small-ball}
Let $\xi_1, \ldots, \xi_n$ be independent centered random variables with variances at least 1 and third moments bounded by $B$. Then for any $\mva \in \bR^n$ and $\eps> 0$, one has 
\[
\sup_{v\in \bR} \bP\Bigl(\Bigl|\sum_{k=1}^n a_k \xi_k - v\Bigr| \le \eps\Bigr) \le \sqrt{\frac{2}{\pi}} \frac{\eps}{\norm{\mva}_2} + C_0 B \Bigl(\frac{\norm{\mva}_3}{\norm{\mva}_2} \Bigr),
\]
where $C_0$ is some absolute constant.
\end{lem}
We also need the following tensorization result.
\begin{lem}[Lemma 2.2~\cite{RV08}]\label{lem:tensor}
    Let $\zeta_1, \ldots \zeta_n$ be independent random variables. Assume that there exist constants $K$ and $\eps_0>0$ such that for each $i$, 
    $\bP(\abs{\zeta_i}< \eps) \le K \eps$ for all $\eps\ge \eps_0$.
     Then there exists an absolute constant $C_1$ such that 
     \[
     \bP\Bigl(\sum_{i=1}^n \zeta_i^2 < \eps^2 n\Bigr) \le (C_1 K\eps)^n\]
     for all $\eps\ge \eps_0$.
\end{lem}
Our goal is to combine Lemma~\ref{thm:small-ball} and Lemma~\ref{lem:tensor} to establish some small-ball probability estimate for the vector $\mvm(\mvgs)$. However, because the matrix $J$ is symmetric, the random variables $m_1(\mvgs), m_2(\mvgs), \ldots, m_N(\mvgs)$ are not independent. To circumvent this obstacle, we will make use of the fact that the entries in any off-diagonal block sub-matrices of $J$ are independent. In particular, we shall work with the sub-matrix $D_2$ of $J$ defined on the block $T\times T^c$, where $T$ is a $(1/(16C), d/4)$-good subset as in Lemma~\ref{lem:T-set}. With this choice of $T$, we partition the vertex set as $[N] = T \cup ([N]\setminus T)$ and accordingly write 
$\mvgs = (\mvgs^{(1)}, \mvgs^{(2)})\in\{-1,+1\}^T\times\{-1,+1\}^{[N]\setminus T}$. Similarly we represent $J$ (after possible permutation of the $N$ vertices) in the following block form,



\begin{figure}
    \centering
    \includegraphics[scale=0.7]{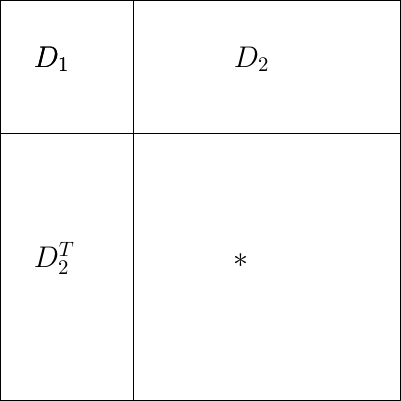}
    \caption{Block decomposition of matrix $J$, where $D_1$ is a $\abs{T}\times\abs{T}$ matrix with row and column index set $T$, while $D_2$ has row index set $T$ and column index set $[N]\setminus T$. It can be seen that the entries in the block $D_2$ are independent from each other.}
    \label{fig:decom}
\end{figure}

\begin{align}\label{eq:D1-D2}
J = \begin{pmatrix}
    D_1 & D_2 \\
    D_2^T & * 
\end{pmatrix},
\end{align}
where $D_1 \in \bR^{\abs{T}\times \abs{T}}$ and $D_2 \in \bR^{\abs{T}\times (N-\abs{T})}$, also see the visualization in Figure~\ref{fig:decom}.
Let $r_i$ be the number of non-zero elements in the $i$-th row of $D_2$.
Note that $r_i$ is same as $\dout_i(T)$ for $i\in T$. It follows from the definition of $T$, 
\begin{align}\label{eq:num-nonzero}
\frac{d}{4} \le r_i \le C d.
\end{align}
Under the above decomposition, 
\[
\quad (D_2 \mvgs^{(2)})_i  = \sum_{j\in [N]\setminus T} J_{ij} \gs_j,\,\,\forall i\in T,
\]
it can be readily seen that the coordinates in $D_2\mvgs^{(2)}$ are now independently from each other. We now use this property and the bound in~\eqref{eq:num-nonzero} to establish the desired small ball probability estimates and use them to further prove the Proposition~\ref{prop:TN-reduced}. 

\begin{proof}[Proof of Proposition~\ref{prop:TN-reduced}]

Recall that $A=(a_{ij})_{i,j}$ is the adjacency matrix of $\mathcal{G}.$
First, as we have explained above that $(D_2\mvgs^{(2)})_{i\in T} = (\la \mva_i ,\mvg_i \ra)_{i\in T}$ 
has independent coordinates and proper moment assumptions, where \begin{align*}
    \mva_i &:= d^{-1/2} (a_{i,\abs{T}+1} \gs_{\abs{T}+1},\ldots, a_{i,N} \gs_{N}),\\
    \mvg&:=(g_{i,|T|+1}, g_{i,|T|+2,}, \ldots, g_{i,N}).
\end{align*}
Recall that \eqref{eq:num-nonzero} implies 
\[
r_i =\abs{\{j\in [N]\setminus T: a_{ij} =1 \}} \in [d/4, C d].
\]
In addition, since
\begin{align*}
    \norm{\mva_i}_2^2 = \frac{r_i}{d}\,\, \text{and} \,\,\norm{\mva_i}_3^3 = \frac{r_i}{d^{3/2}},
\end{align*}
we have 
\[
\norm{\mva_i}_2 \in [1/2, \sqrt{C}]\,\,\text{and}\,\, \Bigl( \frac{\norm{\mva_i}_3}{\norm{\mva_i}_2}\Bigr)^3  
= \frac{1}{\sqrt{r_i}} \le \frac{2}{\sqrt{d}}. 
\]
By Lemma~\ref{thm:small-ball}, we have that for any $i\in T$ and $\eps>0$
\begin{align}\label{eq:small-ball-1d}
\sup_{v\in \bR} \bP\bigl(\bigl|(D_2\mvgs^{(2)})_i - v\bigr|\le \eps\bigr) \le 2\sqrt{\frac{2}{\pi}}\eps + C_0 B \left( \frac{2}{\sqrt{d}}\right)^{1/3} \le K'( \eps + d^{-1/6}),
\end{align}
where $K'$ is a constant only depending on $B$. Note that the hypothesis of Lemma~\ref{lem:tensor} holds for $\eps_0 = d^{-1/6}$ with $K = 2K'$. 
Therefore, for any constant $\gd >0$ and for $d$ sufficiently large such that $\eps_0 = d^{-1/6} \le \delta $, by Lemma~\ref{lem:tensor}, we obtain for any $A \subseteq T$ with $\abs{A} \ge \abs{T}/{2}\ge {N}/(32C)$,
\[
\sup_{\mvv \in \bR^{\abs{A}}} \bP\bigl(\bigl\|(D_2 \mvgs^{(2)})_{i\in A} - \mvv\bigr\|_2^2 \le \gd^2 N\bigr) \le 
(C_1 K \delta)^{N/(32 C)}
\lesssim 10^{-N},
\]
where $C_1$ is as given in the Lemma~\ref{lem:tensor} and the last inequality follows by taking $\delta >0$ sufficiently small constant. 

Recall $\mvm_A(\mvgs) = (m_i(\mvgs))_{i\in A}$. 
Using the previous decomposition of $J$ in~\eqref{eq:D1-D2}, we observe that
\[
\norm{\mvm_A(\mvgs) - \gamma  \mv1_A}_2^2 \ge \bigl\|(D_1 \mvgs^{(1)} + D_2 \mvgs^{(2)})_{i \in A} - \gamma  \mv1_A\bigr\|_2^2.
\]
Then
\begin{align*}
    & \bP\bigl(\norm{\mvm_A(\mvgs) - \gamma \mv1_A}_2^2 \le \gd^2 N \bigr) \\
   \le & \bP\bigl(\bigl\|(D_1 \mvgs^{(1)} + D_2 \mvgs^{(2)})_{i\in A} - \gamma  \mv1_A\bigr\|_2^2\le \gd^2N \bigr) \\
   =  &  \E_{D_1} \Bigl[ \bP \bigl(\bigl\|(D_1 \mvgs^{(1)} + D_2 \mvgs^{(2)})_{i\in A}- \gamma \mv1_A\bigr\|_2^2\le \gd^2N  \big| (D_1\mvgs^{(1)})_{i\in A} \bigr) \Bigr]\\
   \le & \E_{D_1} \Bigl[ \sup_{\mvv \in \bR^{\abs{A}}}  \bP \left(\norm{(D_2\mvgs^{(2)})_{i\in A} - \mvv}_2^2\le \gd^2N  \right)\Bigr] \\
   \le & \sup_{\mvv \in \bR^{\abs{A}}} \bP\left(\norm{(D_2\mvgs^{(2)})_{i\in A} - \mvv}_2^2 \le \gd^2 N\right) \le \frac{K}{10^N}
\end{align*}
for some constant $K$. Since $A$ is arbitrary subset of $T$ with $\abs{A}\ge\abs{T}/2$, we have 
\begin{align*}
\bP\left( \min_{ A\subseteq T: |A| \ge \frac{\abs{T}}{2}}  \norm{\mvm_A(\mvgs) - \gamma\cdot \mv1_A}_{2} \le \gd \sqrt{N}\right) \le \frac{K}{10^N}.
\end{align*}

\end{proof}

\subsection{Proof of Theorem~\ref{thm:TN-check}}\label{ssec:pf-TN-tilde}
In this section, we will complete the proof of Theorem~\ref{thm:TN-check}. First, we prove the following result that states that with high probability, for each $\mvgs$, the variables $m_i(\mvgs)$ are bounded for at least half of the indices in $T$.

\begin{prop}\label{prop:mi-bound}
    Assume $d_{\max} \le C d$. There exist constants $M, K>0$  such that the following event 
   \[  \min_{\mvgs} |\{ i \in T : |m_i(\mvgs)| \le M \}| \ge \frac{|T|}{2}\]
   occurs with probability at least $1 - K \cdot 10^{-N}$.
\end{prop}

\begin{proof}
Consider the matrix $J_T:=(D_1, D_2)$, where $D_1,D_2$ were defined in~\eqref{eq:D1-D2}. Note that for each fixed $\mvgs$,
\[
m_T(\mvgs) = J_T \mvgs = D_1\mvgs^{(1)} + D_2\mvgs^{(2)},
\]
Decompose $D_1$ as 
\[
D_1 = D_1^+ + D_1^-, 
\]
where 
\[
(D_{1}^+)_{ij} := \begin{cases}
    (D_1)_{ij}, & \text{if $i<j$,}\\
     0 , & \text{if $i\ge j$,}
     \end{cases}\,\,\text{and}\,\,(D_{1}^-)_{ij} := \begin{cases}
    (D_1)_{ij}, & \text{if $i>j$,}\\
     0 , & \text{if $i\le j$.}
     \end{cases}
\]
Accordingly, we have the decomposition for $J_T= J_T^+ + J_T^-$ with 
\[
J_T^+:= (D_1^+, D_2)\,\,\text{and}\,\, J_T^-:= (D_1^-, \mv0).
\]
In this way, it can be seen that the entries in $J_T^+$ are independent and so are the entries of $J_T^-$. We let $m_T^+(\mvgs):= J_T^+\mvgs$ and $m_T^-(\mvgs):=J_T^- \mvgs$, whose coordinates are now independent.  For $i\in T$, notice 
\begin{align*}
    \E m_i^{\pm}(\mvgs) = 0\,\,\text{and}\,\,\var(m_i^\pm(\mvgs)) = \E[(m_i^{\pm}(\mvgs))^2] \le  \frac{d_i}{d} \le C.
\end{align*}
The last inequality in the variance bound used the assumption $d_{\max} \le C d$ and the independence among the entries of $J_T^+, J_T^-$. Consequently, from Chebyshev's inequality, we have that for $i \in T$,
\[
\bP\Bigl(\abs{m_i^{\pm}(\mvgs)} \ge \frac{M}{2}\Bigr) \le \frac{4C}{M^2}.
\]
Next, set 
\begin{align*}
    B^{+}_{\mvgs}:= \Bigl\{i\in T,\Big | \abs{m_i^{+}(\mvgs)}\le \frac{M}{2} \Bigr\}\,\,\text{and}\,\, B^{-}_{\mvgs}:= \Bigl\{i\in T,\Big | \abs{m_i^{-}(\mvgs)}\le \frac{M}{2} \Bigr\}.
\end{align*}
Evidently,
\[
\abs{B_{\mvgs}^{+}} \sim \mathrm{Binomial}\left(\abs{T}, p_+ \right)\,\,\text{and}\,\,\abs{B_{\mvgs}^{-}} \sim \mathrm{Binomial}\left(\abs{T},p_-\right),
\]
where $p_-, p_+$ are some constants greater than $1- {4C}/{M^2}$. 
Now, by the binomial concentration,  one can choose sufficiently large $M$ such that 
\[
P\Bigl(\abs{B_{\mvgs}^+}, \abs{B_{\mvgs}^-} \ge \frac45 \abs{T}\Bigr) \ge  1- \frac{K}{20^N}
\]
for some constant $K>0$ independent of $N.$
Let $B_{\mvgs}:=B_{\mvgs}^+ \cap B_{\mvgs}^-$. Using the identity $$|B_{\mvgs}|=|B_{\mvgs}^+|+|B_{\mvgs}^-|-|B_{\mvgs}^+\cup B_{\mvgs}^-|,$$
with probability at least $1-K/  20^{N}$,
\[
\abs{B_{\mvgs}} \ge \frac45 \abs{T} + \frac45 \abs{T} - \abs{T} \ge \frac12 \abs{T},
\]
which readily implies that
\begin{align*}
    \abs{m_i(\mvgs)} \le \abs{m_i^+(\mvgs)} + \abs{m_i^-(\mvgs)} \le \frac{M}{2} + \frac{M}{2} = M
\end{align*}
for $i\in B_{\mvgs}$ and as a consequence,  \[ |\{ i \in T : |m_i(\mvgs)| \le M \}| \geq \frac{|T|}{2}. \]
Finally, applying union bound with resepct to $\mvgs \in \{-1,+1\}^N$ leads to the desired result.
\end{proof}

We proceed to establish the proof of Theorem~\ref{thm:TN-check} via this proposition. 

\begin{proof}[Proof of Theorem~\ref{thm:TN-check}]
For each fixed $\mvgs$, let 
\[ B_{\mvgs} =  \{ i \in T: \abs{m_i(\mvgs)}  \le M\}. \]
By Proposition~\ref{prop:mi-bound}, there exist a constant $M>0$ sufficiently large such that with probability at least $1-K/10^{N}$, 
we have $\min_{\mvgs}\abs{B_{\mvgs}}\ge {\abs{T}}/{2} \ge {N}/{(32C)}$.
Consequently,  for $i\in B_{\mvgs}$
\[
\theta_i(\mvgs) = \sech(\gb m_i(\mvgs)+h) \ge c'':= \sech(\gb M +|h|).
\]
Note that 
\begin{align*}
    \widetilde{T}_N(\mvgs) & \ge \frac{1}{N^2} \sum_{i,j \in B_{\mvgs}} (m_i(\mvgs) - m_j(\mvgs))^2\theta_i(\mvgs) \theta_j(\mvgs) \\
    & \ge \frac{(c'')^2 }{N^2}  \sum_{i,j \in B_{\mvgs}} (m_i(\mvgs) - m_j(\mvgs))^2 \\
    &= \frac{2 (c'')^2 }{N} \sum_{i \in B_{\mvgs}} (m_i(\mvgs) - \bar{m}_{B_{\mvgs}}(\mvgs))^2,
\end{align*}
where $\bar m_{B_{\mvgs}}(\mvgs)$ is the average of $\mvm_{B_{\mvgs}(\mvgs)}(\mvgs).$ 
Therefore, on the event $|B_{\mvgs}| \ge  \abs{T}/{2} $, we have
\begin{align}
 \widetilde{T}_N(\mvgs) &\ge \frac{2 (c'')^2 }{N}  \inf_{ \gamma \in [-M, M]} \sum_{i \in B_{\mvgs}} (m_i(\mvgs) - \gamma)^2  \nonumber \\
&\ge \frac{2 (c'')^2 }{N}  \inf_{ \gamma \in [-M, M]} \min_{ A \subseteq T: |A| \ge |T|/2} \sum_{i \in A} (m_i(\mvgs) - \gamma)^2 \nonumber  \\
&\ge \frac{2 (c'')^2 }{N}  \inf_{ \gamma \in \mathcal{D} } \min_{ A \subseteq T: |A| \ge |T|/2} \Bigl ( \norm{\mvm_A(\mvgs) - \gamma\cdot \mv1}_2   - \frac{\delta \sqrt{N}}{2}  \Bigr)_+^2,  \label{eq:T_N_tilde_lb_net}
\end{align}
where $\mathcal{D}$ is an arbitrary partition of $[-M,M]$ with mash $(\delta/2)$ and $|\mathcal{D}| \le 10 M \delta^{-1}$.
 Applying Proposition~\ref{prop:TN-reduced} for  a fixed $\gamma$ and $\mvgs$ in  \eqref{eq:T_N_tilde_lb_net} and then taking a union bound over $\gamma \in \mathcal{D}$ and $\mvgs $, we obtain, for $c_0  = (c'')^2  \delta^2/2$,
\begin{align*}
\bP\bigl(\min_{\mvgs} \widetilde{T}_N(\mvgs) \le c_0\bigr) \le K10^{- N} + |\mathcal{D}|2^N \cdot  K 10^{-N}.
\end{align*}
The proof of Theorem~\ref{thm:TN-check} now immediately follows.
\end{proof}







\bibliography{est.bib} 

\end{document}

%% file: style.tex
\usepackage{systeme}
\usepackage{mathrsfs,xfrac} 
\usepackage[colorlinks=true,linkcolor=blue,citecolor=blue,urlcolor=blue]{hyperref} 
\usepackage{amsmath,amssymb,amsthm,amsfonts,amsbsy,latexsym,dsfont,color} 
\usepackage[numeric,initials,nobysame]{amsrefs} 
\usepackage[foot]{amsaddr}

\usepackage[utf8]{inputenc}
\usepackage[english]{babel}
\usepackage{comment}

\usepackage[textsize=tiny]{todonotes}
\usepackage{regexpatch}
\makeatletter
\xpatchcmd{\@todo}{\setkeys{todonotes}{#1}}{\setkeys{todonotes}{inline,#1}}{}{}
\makeatother

\usepackage{enumerate}

\newenvironment{enumeratea}{\begin{enumerate}[\upshape (a)]}{\end{enumerate}}

\newtheorem{thm}{Theorem}[section]
\newtheorem{lem}[thm]{Lemma}
\newtheorem{cor}[thm]{Corollary}
\newtheorem{prop}[thm]{Proposition}
\newtheorem{defn}[thm]{Definition}
\newtheorem{rem}[thm]{Remark}

\newtheorem{ass}[thm]{Assumption}

\newtheorem{ex}[thm]{Example}
   
\renewcommand{\le}{\leqslant} 
\renewcommand{\ge}{\geqslant} 
\renewcommand{\leq}{\leqslant} 
\renewcommand{\geq}{\geqslant} 

\newcommand{\ra}{\rangle}
\newcommand{\la}{\langle}

\newcommand{\ind}{\mathds{1}}
\newcommand{\eps}{\varepsilon}

\newcommand{\norm}[1]{\left\Vert#1\right\Vert}
\newcommand{\abs}[1]{\left\vert#1\right\vert}

\newcommand{\ie}{\emph{i.e.,}}

 \let\gb=\beta  \let\gd=\delta 
     \let\gl=\lambda           \let\gs=\sigma \let\gt=\tau

         \let\gS=\Sigma  
                                
\newcommand{\cD}{\mathcal{D}}
\newcommand{\cG}{\mathcal{G}}

\newcommand{\mv}[1]{\boldsymbol{#1}}

\newcommand{\mvX}{\boldsymbol{X}}
\newcommand{\mva}{\boldsymbol{a}}

\newcommand{\mvg}{\boldsymbol{g}}

\newcommand{\mvm}{{m}}

\newcommand{\mvv}{\boldsymbol{v}}
\newcommand{\mvx}{\boldsymbol{x}}

\newcommand{\mvgs}{{\sigma}}

\newcommand{\mvgt}{{\tau}}

\newcommand{\bE}{\mathbb{E}}

\newcommand{\bP}{\mathbb{P}}\newcommand{\bR}{\mathbb{R}}

\DeclareMathOperator{\E}{\mathds{E}}

\DeclareMathOperator{\var}{Var}

\DeclareMathOperator{\sech}{sech}
